\documentclass[10pt]{article}

\usepackage{amsmath, amsthm, amscd, amsfonts,amssymb}
\usepackage[all]{xy}
\usepackage{color}
\usepackage{pgf,tikz}
\usepackage{mathrsfs}
\usepackage[draft]{changes}
\usetikzlibrary{arrows}
\usepackage{graphicx, caption, subcaption}
\usepackage[right = 2cm, top = 2cm, bottom = 2cm, left = 2cm]{geometry}
\usepackage{hyperref}

\newtheorem*{thm*}{Theorem}         
\newtheorem{thm}{Theorem}
\newtheorem{lemma}[thm]{Lemma}

\newtheorem{claim}[thm]{Claim}
\newtheorem{conjecture}[thm]{Conjecture}

\newtheorem*{cor*}{Corollary}
 
\newtheorem*{prop*}{Proposition}

\newtheorem*{mydef*}{Definition}

\newtheorem*{rem}{Remark}

\newtheorem{newcase}{Case}

\definechangesauthor[name=Taylor Short, color=green]{TS}

\title{Every subcubic multigraph is $(1,2^7)$-packing edge-colorable}
\author{
	Xujun Liu\thanks{Department of Foundational Mathematics, Xi'an Jiaotong-Liverpool University, Suzhou, Jiangsu, China, xujun.liu@xjtlu.edu.cn} \and
	Michael Santana\thanks{Department of Mathematics, Grand Valley State University, Allendale, MI, USA, santanmi@gvsu.edu; research of M. Santana was supported by the AMS-Simons Travel Grant} \and
	Taylor Short\thanks{Department of Mathematics, Grand Valley State University, Allendale, MI, USA, shorttay@gvsu.edu}
}

\begin{document}
	\maketitle
	
	\begin{abstract}
		For a non-decreasing sequence $S = (s_1, \ldots, s_k)$ of positive integers, an $S$-packing edge-coloring of a graph $G$ is a decomposition of edges of $G$ into disjoint sets $E_1, \ldots, E_k$ such that for each $1 \le i \le k$ the distance between any two distinct edges $e_1, e_2 \in E_i$ is at least $s_i+1$. The notion of $S$-packing edge-coloring was first generalized by Gastineau and Togni from its vertex counterpart. They showed that there are subcubic graphs that are not $(1,2,2,2,2,2,2)$-packing (abbreviated to $(1,2^6)$-packing) edge-colorable and asked the question whether every subcubic graph is $(1,2^7)$-packing edge-colorable. Very recently, Hocquard, Lajou, and Lu\v zar showed that every subcubic graph is $(1,2^8)$-packing edge-colorable and every $3$-edge colorable subcubic graph is $(1,2^7)$-packing edge-colorable. Furthermore, they also conjectured that every subcubic graph is $(1,2^7)$-packing edge-colorable.
		
		In this paper, we confirm the conjecture of  Hocquard, Lajou, and Lu\v zar, and extend it to multigraphs.
	\end{abstract}
	
	\section{Introduction}
	Given a non-decreasing sequence $S = (s_1, \ldots, s_k)$ of positive integers, an $S$-packing edge-coloring of a graph $G$ is a decomposition of the edge set of $G$ into disjoint sets $E_1, \ldots, E_k$ such that for each $1 \le i \le k$ the distance between any two distinct edges $e_1, e_2 \in E_i$ is at least $s_i+1$ (where the edge-distance in $G$ is defined as the vertex-distance in the line graph of $G$). The notion of $S$-packing edge-coloring was first generalized by Gastineau and Togni~\cite{GT2} from its vertex counterpart, which has also garnered a significant amount of research \cite{BKL1, BKL2, BF1, BKRW1, FKL1, GT1}.
	
	In this paper we focus on the case where each $s_i \in \{1,2\}$.  When $s_i = 1$, we view every edge in $E_i$ as colored with a `1-color', and when $s_i = 2$, we view every edge in $E_i$ as colored with a `2-color'.  For example, suppose a graph $G$ has an $(s_1, \dots, s_k)$-packing edge-coloring, where for some $j, 1 \le j \le k$, we have $s_i = 1$ for all $1 \le i \le j$, and $s_i = 2$ for all $j+1 \le i \le k$.  We would say that such a coloring uses $j$ distinct 1-colors and $k-j$ distinct 2-colors.  Furthermore, we will also call such an $S$-packing edge-coloring, a $(1^j, 2^{k-j})$-packing edge-coloring.
	
	Note that for each 1-color, its color class forms a matching in the graph; whereas for each 2-color, its color class forms an induced matching in the graph.  As a result,  when each $s_i = 1$, an $S$-packing edge-coloring is equivalent to a proper edge-coloring, and when each $s_i = 2$, an $S$-packing edge-coloring is equivalent to what is known as a strong edge-coloring. The strong chromatic index of a graph $G$ is the minimum integer $k$ such that $G$ has a $(2^k)$-packing edge-coloring. Strong edge-colorings were first introduced by Fouquet and Jolivet \cite{FJ} and then extended by many researchers \cite{A1, CDYZ, CKKR, FKS, HQT, HSY, KLRSWY, LMSS, MR1, SY}.   Therefore, one can view $(1^j, 2^{k-j})$-packing edge-colorings, as an intermediate form of coloring between proper edge-colorings and strong edge-colorings. Erd\H{o}s and Ne\v set\v ril~\cite{EN} conjectured that the upper bound for the strong chromatic index of a graph with maximum degree $\Delta$ is $\frac{5}{4} \Delta^2$ when $\Delta$ is even and $\frac{5}{4} \Delta^2 - \frac{1}{2} \Delta + \frac{1}{4}$ when $\Delta$ is odd. In particular, when $\Delta = 3$, Andersen~\cite{A1} and independently Hor\'ak, Qing, and Trotter~\cite{HQT} confirmed the conjecture of Erd\H{o}s and Ne\v set\v ril by showing that every subcubic graph has a $(2^{10})$-packing edge-coloring.  Kostochka, Li, Ruksasakchai, Santana, Wang, and Yu~\cite{KLRSWY} proved that every subcubic planar graph has a $(2^{9})$-packing edge-coloring. It is also well-known by Vizing's Theorem~\cite{V1} that a subcubic graph with maximum degree $3$ either has a $(1^3)$-packing edge-coloring or a $(1^4)$-packing edge-coloring. For graphs with maximum degree at most 4, Huang, Santana, and Yu~\cite{HSY} showed that every graph with maximum degree at most 4 has a $(2^{21})$-packing edge-coloring, which is only one color away from the conjecture of Erd\H{o}s and Ne\v set\v ril which claims that a $(2^{20})$-packing edge-coloring exists.
	
	An immediate corollary of the result ``every subcubic graph has a $(2^{10})$-packing edge-coloring'' is that every subcubic graph has a $(1, 2^9)$-packing edge-coloring.  Recently, Hocquard, Lajou, and Lu\v{z}ar in \cite{HLL2} showed that every subcubic graph has a $(1,2^8)$-packing edge-coloring, and posed the following conjecture, which was initially posed as a question by Gastineau and Togni in \cite{GT2}:
	
	\begin{conjecture}[Hocquard, Lajou, and Lu\v zar~\cite{HLL2}]\label{hll1}
		Every subcubic graph has a $(1,2^7)$-packing edge-coloring.
	\end{conjecture}

	In this paper, we prove that Conjecture~\ref{hll1} is true by proving a slightly stronger statement, utilizing techniques from \cite{A1} and \cite{HSY}.
	
	\begin{thm}\label{LSS1}
		Every subcubic multigraph is $(1,2^7)$-packing edge-colorable so that the 1-color only appears on edges whose endpoints are both 3-vertices. 
	\end{thm}
	
	The following graph $G$ shows that our result in Theorem~\ref{LSS1} is sharp.  It is not $(1,2^6)$-packing edge-colorable as there are 10 edges in $G$,  the maximum size of a matching in $G$ is three, and the maximum size of an induced matching in $G$ is one. Therefore, at most three edges in $G$ can be colored with the 1-color, which leaves the remaining edges (of which there are at least seven) to have a distinct 2-color. Thus $G$ has no $(1, 2^6)$-packing edge-coloring.  However, we can color $x_1x_5, x_4x_7$, and $x_3x_6$ with the 1-color, and then give each of the remaining seven edges a distinct 2-color.  This produces a $(1,2^7)$-packing edge-coloring of $G$ in which the 1-color only appears on edges whose endpoints are both 3-vertices.
	
	\begin{rem}
		If the aforementioned conjecture of Erd\H{o}s and Ne\v set\v ril is true, then it is best possible due to an appropriate blow-up of a five cycle $v_1v_2v_3v_4v_5$ (each vertex in $\{v_1, v_2, v_3\}$ is replaced by $\lfloor \frac{\Delta}{2} \rfloor$ independent vertices and each vertex in $\{v_4, v_5\}$ is replaced by $\lceil \frac{\Delta}{2} \rceil$ independent vertices).  The graph in Figure~\ref{ex-1} is in fact this same graph when $\Delta = 3$.  Furthermore, Hocquard, Lajou, and Lu\v{z}ar also conjectured that every subcubic graph has a $(1^2, 2^4)$-packing edge-coloring.  If true, this same graph serves as a sharpness example in that case as well.
		

	\end{rem}
	
	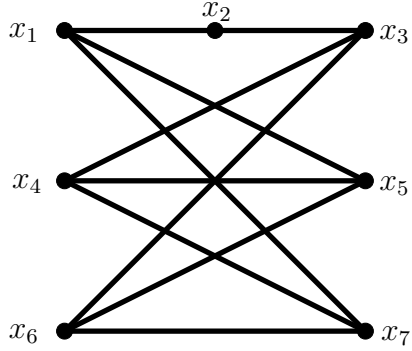
\begin{figure}
		
		\begin{center}
			\begin{tikzpicture}[line cap=round,line join=round,>=triangle 45,x=1.0cm,y=1.0cm]
				\clip(-20.,-11.) rectangle (-12.,-5.);
				
				\draw (-18.864985681280427,-5.781155015702194) node[anchor=north west] {\large $x_1$};
				\draw (-16.3,-5.5) node[anchor=north west] {\large $x_2$};
				\draw (-13.936800630853657,-5.802581907225789) node[anchor=north west] {\large $x_3$};
				\draw (-18.822131898233238,-7.795282818920108) node[anchor=north west] {\large $x_4$};
				\draw (-13.936800630853657,-7.816709710443702) node[anchor=north west] {\large $x_5$};
				\draw (-18.864985681280427,-9.787983730614426) node[anchor=north west] {\large $x_6$};
				\draw (-13.915373739330063,-9.80941062213802) node[anchor=north west] {\large $x_7$};
				\draw (-8.79434666519094,-2.1171565651674795) node[anchor=north west] {\large $x_8$};
				\draw (-7.7230020890112066,-2.1171565651674795) node[anchor=north west] {\large $x_9$};
				\draw (-6.6088037297842845,-2.1171565651674795) node[anchor=north west] {\large $x_{10}$};
				
				\draw [line width=2.pt] (-18.,-6.)-- (-14.,-6.);
				\draw [line width=2.pt] (-18.,-8.)-- (-14.,-6.);
				\draw [line width=2.pt] (-14.,-6.)-- (-18.,-10.);
				\draw [line width=2.pt] (-18.,-10.)-- (-14.,-10.);
				\draw [line width=2.pt] (-14.,-10.)-- (-18.,-8.);
				\draw [line width=2.pt] (-18.,-8.)-- (-14.,-8.);
				\draw [line width=2.pt] (-14.,-8.)-- (-18.,-6.);
				\draw [line width=2.pt] (-18.,-6.)-- (-14.,-10.);
				\draw [line width=2.pt] (-18.,-10.)-- (-14.,-8.);
				\begin{scriptsize}
					\draw [fill=black] (-18.,-6.) circle (3.0pt);
					\draw [fill=black] (-14.,-6.) circle (3.0pt);
					\draw [fill=black] (-18.,-8.) circle (3.0pt);
					\draw [fill=black] (-18.,-10.) circle (3.0pt);
					\draw [fill=black] (-14.,-10.) circle (3.0pt);
					\draw [fill=black] (-14.,-8.) circle (3.0pt);
					\draw [fill=black] (-16.,-6.) circle (3.0pt);
				\end{scriptsize}
			\end{tikzpicture}
			
		\end{center}
		\caption{Sharpness example of Theorem~\ref{LSS1}.}\label{ex-1}
		
	\end{figure}
	
	The structure of the rest of this paper is as follows.  In Section \ref{sec:lemmas} we show that a minimal counterexample to our statement  is 3-regular, contains no edge cuts of size at most two, and has girth at least seven.  In Section \ref{sec:cuts} we finish the proof of Theorem \ref{LSS1} by mimicking proofs in \cite{A1} and \cite{HSY} regarding edge cuts in our graph.  We then  present several questions for further consideration in Section \ref{sec:further}.  Lastly, after the references we have an Appendix that contains links to two pieces of code that were used to prove two lemmas.

	\section{Main lemmas}\label{sec:lemmas}
	
	The main goal of this section is to show that a minimal counterexample to Theorem \ref{LSS1} is 3-regular, contains no edge cuts of size at most two, and has girth at least seven.  We first introduce some terminology and notation that will be used throughout the paper.
	
	We will use the color 1 to denote the 1-color, and use the colors $2_1, 2_2, \dots, 2_7$ to denote the seven distinct 2-colors in a potential $(1,2^7)$-packing edge-coloring of a graph.
	
	Let $e$ and $e'$ be two edges in a graph.  We say $e$ sees $e'$ (and vice-versa) if $e$ and $e'$ are incident to each other, or incident to a common edge.
	

		Let $G$ be a graph, and let $f$ be a coloring of some (possibly all) the edges of $G$.  For a vertex $v \in V(G)$, let $U_f(v) = \{ f(e): e \text{ is incident to } v\}$.

		For an edge $e \in E(G)$, let $S_f(e) = \{f(e') : e \text{ sees } e'\}$, and  let $S^2_f(e) = \{2_i : e \text{ sees an edge colored with } 2_i\}$.
		
		For an edge $e \in E(G)$, let $A_f(e) = \{1, 2_1, \dots, 2_7\} \setminus S_f(e)$, and  let $A^2_f (e) = \{2_1, \dots, 2_7\}\setminus S^2_f(e)$. 
		
		We often drop the subscript  indicating the coloring if it is clear from context.
		
		A good coloring of a graph is a $(1,2^7)$-packing edge-coloring of the graph such that the color 1 only appears on edges whose endpoints are both 3-vertices.  A good partial coloring of a graph is a coloring of some (possibly all) the edges of $G$ using the colors $1, 2_1, \dots, 2_7$, such that the color 1 only appears on edges whose endpoints are both 3-vertices, and if two edges both colored with the same 2-color, then they do not see each other.
		
		Often in our arugments we will compare the sizes of $A^2(e)$ for various uncolored edges.  We will then use a formulation of Hall's Theorem to obtain a good coloring of a graph via a system of distinct representatives.  In such a case, we will say statements such as `We are done by SDR'.
		
		In all the following, let $G$ be a counterexample to Theorem \ref{LSS1} such that $|V(G)| + |E(G)|$ is smallest.  As a result, we may assume that $G$ is connected by the minimality of $G$.

		\begin{lemma}
			$G$ is simple.
		\end{lemma}
		
		\begin{proof}
			Suppose $G$ has two vertices $x$ and $y$ such that there are multiple edges between $x$ and $y$.  If there are three edges between $x$ and $y$, then these are the only vertices and edges in $G$, and there is easily a good coloring of $G$.
			
			If there are exactly two edges between $x$ and $y$, say $e_1$ and $e_2$, then by the minimality of $G$, $G - e_1$ has a good coloring.  If we impose this coloring onto $G$, then $e_1$ is the only edge that remains to be colored.  
			
			Note that the maximum number of edges $e_1$ sees in $G$ is seven.  So the only way  $|A^2(e_1)| = 0$ is if $x$ and $y$ are both 3-vertices in $G$, and all seven 2-colors are used on those  edges seen by $e_1$.  This means we can color $e_1$ with 1 to obtain a good coloring of $G$.
		\end{proof}

		\begin{lemma}\label{lem:delta2}
			$\delta(G) \ge 2$.
		\end{lemma}
		
		\begin{proof}
			Suppose on the contrary that $G$ contains a vertex $v$ whose only neighbor is a vertex $x$.   By the minimality of $G$, $G - v$ has a good coloring.  If we impose this coloring onto $G$, then $vx$ is the only uncolored edge.
			
			Since $vx$ sees at most six edges, $|A^2(vx)| \ge 1$, so we can color $vx$.
		\end{proof}

		\begin{lemma}\label{lem:cutedge}
			$G$ has no cut-edge, and as a consequence, $G$ has no cut-vertex.
		\end{lemma}

		\begin{proof}
			Note that since $\Delta(G) \le 3$, if $G$ has a cut-vertex, then $G$ must have a cut-edge.  So it suffices to only consider when $G$ has a cut-edge.
			
			Suppose on the contrary that $x_1x_2$ is a cut-edge in $G$.  Let $G_1$ and $G_2$ be the two components of $G - x_1x_2$ such that $x_i$ is in $G_i$.  From each $G_i$ form a new graph $H_i$ by adding a new vertex $y_i$ and the edge $x_iy_i$.
			
			By Lemma \ref{lem:delta2}, $x_i$ has another neighbor in $H_i$ other than $y_i$, so that $H_i$ has fewer vertices than $G$.  Thus, by the minimality of $G$, $H_i$ has a good coloring, say $\phi_i$.
			
			Note that $x_iy_i$ must get a 2-color in $\phi_i$.  We can permute the 2-colors on both $H_1$ and $H_2$ so that $x_1y_1$ and $x_2y_2$ are colored with $2_1$.  We can furthermore, permute the 2-colors on $H_1$ so that $U_{\phi_1}(x_1) \subseteq \{2_1, 2_2, 2_3, 1\}$ and we can permute the 2-colors on $H_2$ so that $U_{\phi_2}(x_2) \subseteq \{2_1, 2_4, 2_5, 1\}$.  
			
			Imposing these colorings onto $G$ yields a good coloring of $G$.
		\end{proof}

		\begin{lemma}\label{lem:cut}
			If $G$ has an edge-cut of size two, then it must consist of two edges incident to the same 2-vertex.
		\end{lemma}
		
		\begin{proof}
			Suppose $e_1$ and $e_2$ form an edge-cut of size two.  By Lemma \ref{lem:cutedge}, neither one is a cut-edge so that $G - e_1 - e_2$ has exactly two components.  If $e_1$ and $e_2$ are incident to one another, then either their common endpoint is a 2-vertex or a 3-vertex.  If it is a 2-vertex, we are done.  If it is a 3-vertex, then the third edge incident to this 3-vertex is a cut-edge, contradicting Lemma \ref{lem:cutedge}.  So we may assume  $e_1$ and $e_2$ form a matching, and let $e_1 = x_1x_2$ and let $e_2 = y_1y_2$ such that in $G - e_1 - e_2$,  $x_i$ and $y_i$ are in the same component, say $G_i$.
			
			
			Suppose both $x_1$ and $x_2$ are 3-vertices.  For each $i \in [2]$ form $H_i$ from $G_i$ by adding a vertex $z_i$ adjacent to only $y_i$.  By the minimality of $G$, each $H_i$ has a good coloring where $U(x_i)$ contains only 2-colors and $y_iz_i$ is not colored with color 1.  So we can permute these colorings so that $y_iz_i$ is colored $2_7$ for each $i \in [2]$, we can furthermore permute the colors so that $U(x_1) \subseteq \{2_1,2_2,2_7\}, U(y_1) \subseteq \{1, 2_1,2_2,2_3,2_4, 2_7\}$, $U(y_2) \subseteq \{1,2_5,2_6,2_7\}$, and $U(x_2) = \{2_3,2_4,2_5,2_6,2_7\}$.  We then impose these colorings onto $G$ by coloring $y_1y_2$ with $2_7$ and $x_1x_2$ with 1.

			In all the rest of the cases, we modify each $G_i$ as follows. For each $i \in [2]$, form $H_i$ from $G_i$ by adding a vertex $z_i$ adjacent to both $x_i$ and $y_i$.  By the minimality of $G$, each $H_i$ has a good coloring where $x_iz_i$ and $y_iz_i$ are colored with 2-colors.  So we can permute these colorings so that $x_iz_i$ is colored with  $2_6$ and $y_iz_i$ is colored with $2_7$.  Note that as a result, $2_6 \notin U(y_1) \cup U(y_2)$ and $2_7 \notin U(x_1) \cup U(x_2)$.
			
			Now suppose both $x_1$ and $x_2$ are 2-vertices in $G$.  By the above and symmetry, we may assume that $y_2$ is a 2-vertex in $G$.    So in $H_1$ and $H_2$, we can furthermore permute the colors so that $U(x_1) = \{2_1, 2_6\}$, $U(y_1) \subseteq \{1, 2_1, 2_2, 2_3, 2_7\}$, $U(y_2) = \{2_4,2_7\}$, and $U(x_2) \subseteq \{2_4, 2_5, 2_6\}$.  We then impose these colorings onto $G$ by coloring $x_1x_2$ with $2_6$ and $y_1y_2$ with $2_7$.
			
			So we may assume that $x_1$ is a 3-vertex, and $x_2$ is a 2-vertex in $G$.  If $y_1$ is a 3-vertex and $y_2$ is a 2-vertex in $G$, then we can further permute the colors in each $H_i$ so that $U(x_1) \subseteq \{1, 2_1, 2_2, 2_6\}$, $U(y_1) \subseteq \{1, 2_1, 2_2, 2_3, 2_4, 2_7\}$, $U(y_2) \subseteq \{2_5, 2_7\}$, and $U(x_2) \subseteq \{2_4, 2_5, 2_6\}$.  We then impose these colorings onto $G$ by coloring $x_1x_2$ with $2_6$ and $y_1y_2$ with $2_7$.
			
			Lastly, if $y_1$ is a 2-vertex and $y_2$ is a 3-vertex in $G$, then we can further permute the colors in each $H_i$ so that $U(x_1) \subseteq \{1, 2_1, 2_2, 2_6\}$, $U(y_1) \subseteq \{2_1, 2_2, 2_3, 2_7\}$, $U(y_2) \subseteq \{1, 2_4, 2_5, 2_7\}$, and $U(x_2) \subseteq \{2_3, 2_4, 2_5, 2_6\}$.  We then impose these colorings onto $G$ by coloring $x_1x_2$ with $2_6$ and $y_1y_2$ with $2_7$.
		\end{proof}

		\begin{lemma}\label{lem:noK3}
			$G$ has no cycle of length 3.
		\end{lemma}
		
		\begin{proof}
			Suppose on the contrary that $x_1x_2x_3x_1$ is a cycle of length 3.  For each $i \in [3]$, let $y_i$ be a neighbor of $x_i$ not on the cycle (if it exists).
			
			We first show that each $x_i$ is a 3-vertex.  Clearly, not all are 2-vertices otherwise $G = K_3$ and we easily provide a good coloring of $G$.  If only say $x_3$ is a 3-vertex, then $x_3y_3$ is a cut-edge contradicting Lemma \ref{lem:cutedge}.  If say $x_1$ is a 2-vertex, then   $x_2y_2, x_3y_3$ form an edge-cut of size two, so by Lemma \ref{lem:cut}, $G = K_4^-$ and we can easily provide a good coloring of $G$.
			
			Clearly $|\{y_1,y_2,y_3\}| \ge 2$ otherwise $G = K_4$ and we can easily provide a good coloring of $G$.  If say $y_2 = y_3$, then $y_2$ must be a 3-vertex, otherwise $x_1y_1$ is a cut-edge contradicting Lemma \ref{lem:cutedge}.  However, $x_1y_1$ and the edge incident to $y_2$ other than $x_2y_2, x_3y_2$, form an edge-cut of size two, so by Lemma \ref{lem:cut}, $G$ has at most 5 vertices and we can easily provide a good coloring of $G$.  So the $y_i$'s all exist and are all distinct.

			\begin{claim}
				Each $y_i$ is a 3-vertex.
			\end{claim}
			
			\begin{proof}
				Suppose $y_1$ is a 2-vertex where $N(y_1) = \{x_1,z_1\}$.  Then $G - x_1$ has a good coloring by the minimality of $G$.  Note that $U(y_1), U(x_2)$, and $U(x_3)$ only contain 2-colors.  So we can impose this coloring onto $G$ and color $x_1x_2$ with a 1 to give a good partial coloring of $G$, say $\phi$.  Note that $|A_\phi(x_1y_1)| \ge 1$ and $|A_\phi(x_1x_3)| \ge 1$.  So we may assume equality holds and $A_\phi(x_1y_1) = A_\phi(x_1x_3) = \{2_7\}$.  So without loss of generality, $\phi(y_1z_1) = 2_1$, $U_\phi(z_1) = \{2_1, 2_2, 2_3\}$, $U_\phi(x_2) = \{2_4, 2_5\}$, and $U_\phi(y_2) = \{2_2,2_3,2_4\}$.  Note $x_2y_2$ and $x_2x_3$ are colored with $2_4$ and $2_5$, respectively, which also means $z_1 \neq y_2$ as $y_1z_1$ is colored with $2_1$.
				
				However, if we form a new good partial coloring of $G$ by uncoloring $x_1x_3$ and instead coloring $x_1x_2$ with 1 (call this coloring $\psi$), then as above we must have $A_\psi(x_1y_1) = A_\psi(x_1x_3) = \{2_7\}$.  So $U_\psi(y_3) = \{2_2, 2_3, 2_6\}$.   Note that $y_3$ must be a 3-vertex that is not $z_1$ (as $\psi(y_1z_1) = 2_1$), so we recolor $x_3y_3$ with 1, color $x_1y_1$ with $2_7$ and color $x_1x_3$ with $\psi(x_3y_3) \in \{2_2,2_3,2_6\}$.  So each $y_i$ is a 3-vertex.
			\end{proof}
			
			By the minimality of $G$, $G - x_1 - x_2 - x_3$ has a good coloring in which each $y_i$ is not incident to edges colored with 1.  So we can impose this coloring onto $G$ and color each $x_iy_i$ with 1.  Since $|U(x_1x_2)|, |U(x_2x_3)|, |U(x_1x_3)| \ge 3$, we can greedily color the remaining edges.
		\end{proof}

		\begin{lemma}\label{nine}
			$G$ must have at least 10 vertices.
		\end{lemma}
		
		\begin{proof}
			The proof of this lemma utilizes SAGE.  A link to the proof can be found in the Appendix.
		\end{proof}

		\begin{lemma}\label{lem:no2onC4}
			$G$ has no 4-cycle with a 2-vertex on it.
		\end{lemma}
		
		\begin{proof}
			Let $x_1x_2x_3x_4x_1$ be a 4-cycle in $G$, and let $y_i$ be neighbor of $x_i$ not on this cycle (if it exists).  Suppose $x_1$ is a 2-vertex. If any other $x_i$ is also a 2-vertex, then by symmetry we have two cases, i.e., $x_2$ is a 2-vertex and $x_3$ is a 2-vertex. In the former case, $x_1x_4$ and $x_2x_3$ form an edge-cut of size two, which contradicts Lemma~\ref{lem:cut}. In the latter case, say the neighbour of $x_2$ and $x_4$ outside of the cycle is $y_2$ and $y_4$ respectively, then by Lemma~\ref{lem:cut} we know that $y_2 = y_4$ and it is a $2$-vertex; however, we then have a graph with only 5 vertices and can easily provide a good coloring by Lemma~\ref{nine}. Therefore, $y_2, y_3, y_4$ exist.
			
			Clearly $|\{y_2,y_3,y_4\}| \ge 2$ otherwise $G$ would be a subgraph of the wheel on 5 vertices and we can easily provide a good coloring of $G$ by Lemma~\ref{nine}. If say $y_2 = y_3$, then as above, we can either find a cut-edge contradicting Lemma \ref{lem:cutedge}, or an edge-cut of size two so that by Lemma \ref{lem:cut} $G$ has only 6 vertices and we can easily provide a good coloring of $G$ by Lemma~\ref{nine}.  So the $y_i$'s all exist and are distinct.
			
			By the minimality of $G$, $G - x_1$ has a good coloring.  If we impose this coloring onto $G$, then note that both $x_1x_2$ and $x_1x_4$ see at most seven colored edges.  Furthermore, the only edge with both endpoints in $\{x_1, \dots, x_4, y_1, \dots, y_4\}$ that can be colored 1 is $x_3y_3$.  Suppose first that $|A^2(x_1x_2)| = 0$.  So without loss of generality, $U(y_1) = \{2_1,2_2,2_3\}$, $U(x_2) = \{2_1,2_4\}$, $U(x_3) = \{2_4,2_5,2_6\}$, and $U(x_4) = \{2_5,2_7\}$.  In particular, $y_2$ must be a 3-vertex, $x_1x_2$ is colored with $2_1$ and $x_3x_4$ is colored with $2_5$.  Therefore we can recolor $x_2y_2$ and $x_3x_4$ both with 1, and then color $x_1x_2$ with $2_1$ and $x_1x_4$ with $2_5$.  Since $x_3x_4$ was originally colored with $2_5$ in $G - x_1$, we know that $2_5 \notin U(y_4)$.  So this is a good coloring of $G$.
			
			So we may assume by symmetry that $|A^2(x_1x_2)|, |A^2(x_1x_4)| \ge 1$.   Further, we may assume without loss of generality that equality holds and $A^2(x_1x_2) = A^2(x_1x_4) = \{2_7\}$, else we are done by SDR.  Recall  the only edge with both endpoints in $\{x_1, \dots, x_4, y_1, \dots, y_4\}$ that can be colored 1 is $x_3y_3$
			
			If $x_3y_3$ is colored with 1, then $|U(y_2)| = 3$ and consists only of 2-colors otherwise  $|A^2(x_1x_2)| \ge 2$.  So $y_2$ must be a 3-vertex, and we can recolor $x_2y_2$ with 1 to obtain a new, good partial coloring of $G$, say $\phi$.  However $|A^2_\phi(x_1x_4)|, |A^2_\phi(x_1x_2)| \ge 2$, and we are done by SDR.  
			
			So $x_3y_3$ is not colored with 1.  We can then recolor $x_3x_4$ with 1 to obtain a new, good partial coloring of $G$.  However, we repeat the same argument for when $x_3y_3$ is colored with 1.  This finishes the proof of the lemma.
		\end{proof}

		\begin{lemma}\label{4cycle}
			$G$ has no 4-cycle.
		\end{lemma}
		
		\begin{proof}
			Let $x_1x_2x_3x_4x_1$ be a 4-cycle $C$ in $G$, and let $y_i$ be neighbor of $x_i$ not on this cycle.  By Lemma \ref{lem:no2onC4} each $y_i$ exists.  
			
			\begin{claim}\label{alldiff}
				The $y_i$'s are all distinct.
			\end{claim}
			
			\begin{proof}
				Since $G$ is subcubic, $|\{y_1,y_2,y_3,y_4\}| \ge 2$.  Further, we claim that $|\{y_1,y_2,y_3,y_4\}| \ge 3$. Otherwise, by Lemma~\ref{lem:noK3}, we must have  $y_1 = y_3$ and $y_2 = y_4$. If $y_1y_2$ is an edge then we have a graph on $6$ vertices and thus we can easily provide a good coloring by Lemma~\ref{nine}; if $y_1y_2$ is not an edge then we have an edge-cut of size two and by Lemma~\ref{lem:cut} $y_1$ and $y_2$ must have a common neighbour of degree two, which gives us the graph in Figure~\ref{ex-1}. If $|\{y_1,y_2,y_3,y_4\}| = 3$, then by Lemma \ref{lem:noK3}, we may assume $y_2 = y_4$. 
				
				We have two cases depending on whether or not $y_1y_3$ exists. 
				
				\textbf{Case 1:} $y_1y_3$ exists. Note that each $y_i$ must be a 3-vertex. Otherwise, if $y_2$ is a $2$-vertex, then $x_1y_1$ and $x_3y_3$ form an edge-cut of size two, which contradicts Lemma~\ref{lem:cut} since we know $y_1 \neq y_3$. If $y_3$ (by symmetry $y_1$) is a $2$-vertex, let $N(y_1) = \{x_1, y_3, w_1\}$ and $N(y_2) = \{x_2, x_4, w_2\}$,  we have by Lemma \ref{lem:cut} $w_1 = w_2$ and by Lemma~\ref{nine} we can easily provide a good coloring for the resulting graph of $8$ vertices.
				
				By the minimality of $G$, $G - V(C)$ has a good coloring in which $U(y_i)$ consists of only 2-colors for each $i$.  If we impose this coloring onto $G$, then we can color $x_1y_1, x_3y_3$, and $x_4y_2$ with 1.  In this good partial coloring we have $$|A^2(x_1x_2)|, |A^2(x_2x_3)|, |A^2(x_3x_4)|, |A^2(x_1x_4)|, |A^2(x_2y_2)| \ge 4.$$ So we must have equality and each of these sets must be the exact same, else we are done by SDR.  Without loss of generality, we may assume they are $\{2_4, 2_5, 2_6, 2_7\}$. Further, since $x_1x_2$ needs to see $2_1, 2_2, 2_3$, we may assume $U(y_1) = \{1, 2_1, 2_2\}$, $U(y_2) = \{1, 2_3\}$, and since $x_3x_4$ also has to see $\{2_1, 2_2, 2_3\}$, we assume $U(y_3) = \{1, 2_1,2_2\}$. However, this cannot happen as $y_1y_3$ is an edge in $G$.
				
				\textbf{Case 2:} $y_1y_3$ does not exist. By the minimality of $G$, $G - V(C) + y_1y_3$ has a good coloring $\phi$.  If $y_1y_3$ is colored with 1, then we can impose this coloring onto $G$ by coloring $x_1y_1$ and $x_3y_3$ with 1. Similarly to Case 1, we may assume $U(y_1) = \{1, 2_1, 2_2\}$, $U(y_2) = \{1, 2_3\}$, $U(y_3) = \{1, 2_1, 2_2\}$. However, this is also not possible since $\phi$ was a good coloring of $G - V(C) + y_1y_3$.  So we may assume $y_1y_3$ is colored with some color that is not $1$, say $2_1$.  We then impose this coloring onto $G$ by coloring $x_1y_1$ and $x_3y_3$ with $2_1$, and $x_1x_2$ and $x_3x_4$ with 1.  In this good partial coloring of $G$ we have $|A^2(x_2x_3)|, |A^2(x_1x_4)|, |A^2(x_2y_2)|, |A^2(x_4y_2)| \ge 3$.  So we must have equality and each of these sets must be exactly the same, else we are done by SDR.  Without loss of generality, we may assume they are $\{2_5, 2_6, 2_7\}$, and further $U(y_1) = \{2_1, 2_2, 2_3\}, U(y_2) = \{2_4\}$, and $U(y_3) = \{2_1, 2_2, 2_3\}$.  However this cannot happen as $y_1y_3$ was added in $G - V(C) + y_1y_3$ to obtain $\phi$.
				\end{proof}
			
			\begin{claim}
				Not both of $y_1y_3$ and $y_2y_4$ are edges in $G$.
			\end{claim}
			
			\begin{proof}
				Suppose on the contrary that $y_1y_3$ and $y_2y_4$ are edges in $G$.  Note that at least three of the $y_i$'s must be 3-vertices as otherwise we get a cut-edge contradicting Lemma \ref{lem:cutedge}, or we have an edge-cut of size two so by Lemma \ref{lem:cut}, $G$ has only 9 vertices and we can easily provide a good coloring by Lemma \ref{nine}.

				We claim that each $y_i$ is a 3-vertex. Suppose not, say $y_4$ is a 2-vertex. By the minimality of $G$, $G - V(C)$ has a good coloring in which each $U(y_i)$ only contains 2-colors.  We impose this coloring onto $G$ and color $x_iy_i$ with 1 for each $i \in [3]$.  Then in this good partial coloring of $G$ we have $|A^2(x_1x_2)|, |A^2(x_2x_3)| \ge 3$, $|A^2(x_3x_4)|, |A^2(x_1x_4)| \ge 4$, and $|A^2(x_4y_4)| \ge 5$, and we are done by SDR.
				
				So we may assume each $y_i$ is a 3-vertex.  Again by the minimality of $G$, $G - V(C)$ has a good coloring in which each $U(y_i)$ only contains 2-colors.  We impose this coloring onto $G$ and color $x_iy_i$ with 1 for each $i \in [4]$.  Then in this good partial coloring of $G$ we have $|A^2(x_1x_2)|, |A^2(x_2x_3)|, |A^2(x_3x_4)|, |A^2(x_1x_4)| \ge 3$.  So we must have equality and each of these sets must be the exact same, else we are done by SDR. Without loss of generality, we may assume they are $\{2_5,2_6,2_7\}$, and further $U(y_1) = U(y_3) = \{1,2_1,2_2\}$ and $U(y_2) = U(y_4) = \{1,2_3,2_4\}$.  However this cannot happen as $y_1y_3$ and $y_2y_4$ are edges in $G$.  This proves the claim.
			\end{proof}

			Without loss of generality, $y_2y_4$ does not exist. By the minimality of $G$, the graph $G' = G - V(C) + y_2y_4$ has a good coloring, say $\phi$, in which $U_\phi(y_1)$ and $U_\phi(y_3)$ consist only of 2-colors.  Impose this coloring onto $G$, ingoring for the moment, the color on $\phi(y_2y_4)$ in $G'$.   So the only edges in $G$ that are uncolored are the eight edges incident to at least one $x_i$.
			
			We now consider how many $x_iy_i$'s can we color with 1 to extend our good partial coloring of $G$.   Note that although $U_\phi(y_1)$ and $U_\phi(y_3)$ consist only of 2-colors, we may not be able to color either $x_1y_1$ or $x_3y_3$ with 1, as $y_1$ or $y_3$ (or both) may be 2-vertices in $G$.  So the only reason we cannot color $x_iy_i$ with 1 for $i \in \{1,3\}$ is because $y_i$ is a 2-vertex.  If we cannot color $x_iy_i$ with 1 for $i \in \{2,4\}$, then it is either because $y_i$ is a 2-vertex or because $1 \in U_\phi(y_i)$.
			
			\textbf{Case 1:}
			If we can color each $x_iy_i$ with 1, then in this good partial coloring, say $\psi$, we have $$|A^2_\psi(x_1x_2)|, |A^2_\psi(x_2x_3)|, |A^2_\psi(x_3x_4)|, |A^2_\psi(x_1x_4)| \ge 3.$$  So we must have equality and each of these sets must be the exact same, else we are done by SDR.  So without loss of generality, we may assume they are $\{2_5,2_6,2_7\}$, and further $U_\psi(y_1) = U_\psi(y_3) = \{1,2_1,2_2\}$ and $U_\psi(y_2) = U_\psi(y_4) = \{1,2_3,2_4\}$.  However this cannot happen as we added the edge $y_2y_4$ in $G'$.
			
			\textbf{Case 2:}
			Suppose we can color exactly three of the $x_iy_i$'s with 1.  We do so and call this good partial coloring $\rho$.  Without loss of generality, either $x_2y_2$ or $x_3y_3$ cannot be colored with 1.  
			
			\textbf{Case 2.1:} $x_2y_2$ cannot be colored with 1. This is either because $y_2$ is a 2-vertex or because $1 \in U_\phi(y_2)$.  Then $|A^2_\rho(x_1x_2)|,|A^2_\rho(x_2x_3)| \ge 4$, and $|A^2_\rho(x_3x_4)|, |A^2_\rho(x_4x_1)| \ge 3$, and $|A^2_\rho(x_2y_2)| \ge 2$.  So the union of all five of these sets must have cardinality at most 4, else we are done by SDR.  In particular, we may assume $A^2_\rho(x_1x_2) = A^2_\rho(x_2x_3) = \{2_4, 2_5, 2_6, 2_7\}$ so that without loss of generality $U_\rho(y_1) = \{1,2_1,2_2\}$ and $U_\rho(y_2) \subseteq \{1, 2_3\}$.  However since $y_2y_4$ was an edge in $G'$, $2_3 \notin U_\rho(y_4)$, which means that $2_3 \in A^2_\rho(x_1x_4)$ and $|A^2_\rho(x_1x_4) \cup A^2_\rho(x_1x_2)| \ge 5$, and we are done by SDR.
			
			\textbf{Case 2.2:} $x_3y_3$ cannot be colored with 1. Then $y_3$ must be a 2-vertex in $G$.  Then $|A^2_\rho(x_1x_2)|,|A^2_\rho(x_1x_4)| \ge 3$, and $|A^2_\rho(x_2x_3)|, |A^2_\rho(x_3x_4)|,  |A^2_\rho(x_3y_3)| \ge 4$.  So we may assume $A^2_\rho(x_2x_3) = A^2_\rho(x_3x_4) = \{2_4, 2_5, 2_6, 2_7\}$, else we are done by SDR.  However, this implies $U_\rho(y_2) = U_\rho(y_4) =  \{1, 2_1, 2_2\}$ and $U_\rho(y_3) = \{2_3\}$, which cannot happen as we added $y_2y_4$ in $G'$.
			
			\textbf{Case 3:}
			Suppose that we can only color two of the $x_iy_i$'s with 1.  We do so and call this good partial coloring $\sigma$.  Without loss of generality either both $x_1y_1$ and $x_4y_4$, or both $x_2y_2$ and $x_4y_4$, or both $x_1y_1$ and $x_3y_3$ are the only edges that can be colored 1.

			\textbf{Case 3.1:} Only $x_1y_1$ and $x_4y_4$ can be colored with 1.  Note that $y_3$ must be a 2-vertex and $y_2$ is either a 2-vertex or $1 \in U_\phi(y_2)$.  We can also further extend $\sigma$ by coloring $x_2x_3$ with 1, and we will refer to this new coloring as $\sigma$.  Note that $|A^2_\sigma(x_1x_2)|, |A^2_\sigma(x_3x_4)|, |A^2_\sigma(x_3y_3)| \ge 4$, $|A^2_\sigma(x_1x_4)| \ge 3$, and $|A^2_\sigma(x_2y_2)| \ge 2$.  So the union of the five sets must cardinality at most 4, else we are done by SDR.  In particular, we may assume $A^2_\sigma(x_1x_2) = \{2_4, 2_5, 2_6, 2_7\}$, so that without loss of generality, $U_\sigma(y_1) = \{1, 2_1, 2_2\}$ and $2_3 \in U_\sigma(y_2)$.  However since $y_2y_4$ was an edge in $G'$, $2_3 \notin U_\sigma(y_4)$ so that $2_3 \in A^2_\sigma(x_1x_4)$ and $|A^2_\sigma(x_1x_2) \cup A^2_\sigma(x_1x_4)| \ge 5$ and we are done by SDR.
			
			\textbf{Case 3.2:} Only $x_2y_2$ and $x_4y_4$ can be colored 1.  Note that $y_1$ and $y_3$ must both be 2-vertices, and furthermore $y_1y_3$ is not an edge of $G$ as otherwise $x_2y_2$ and $x_4y_4$ would form an edge-cut of size two, so by Lemma \ref{lem:cut} $G$ would only have 7 vertices and we can easily provide a good coloring of $G$ by Lemma~\ref{nine}. Note that $|A^2_\sigma(x_1y_1)|, |A^2_\sigma(x_3y_3)| \ge 4$.  Since we have only seven 2-colors available, there must be a 2-color available on both $x_1y_1$ and $x_3y_3$, say $2_1$.  So we can further extend $\sigma$ by coloring these edges with $2_1$, and continue to refer to this good partial coloring as $\sigma$.  Note that $|A^2_\sigma(x_1x_2)|, |A^2_\sigma(x_2x_3)|, |A^2_\sigma(x_3x_4)|, |A^2_\sigma(x_1x_4)| \ge 3$. So we must have equality and each of these sets must be the exact same, else we are done by SDR.  So without loss of generality, we may assume they are $\{2_5,2_6,2_7\}$, and further $U_\sigma(y_1) = U_\sigma(y_3) =  \{2_1, 2_2\}$ and $U_\sigma(y_2) = U_\sigma(y_4) = \{1, 2_3,2_4\}$.  However this cannot happen as $y_2y_4$ was an edge in $G'$.
			
			\textbf{Case 3.3:} Only $x_1y_1$ and $x_3y_3$ can be colored with 1.  Note that for each $i \in \{2,4\}$ either $y_i$ is a 2-vertex of $1 \in U_\phi(y_i)$, and furthermore $\phi(y_2y_4)$ must have been a 2-color, say $2_1$.  So we can further extend $\sigma$ by coloring $x_2y_2$ and $x_4y_4$ with $2_1$, and continue to refer to this good partial coloring as $\sigma$.  Note that $|A^2_\sigma(x_1x_2)|, |A^2_\sigma(x_2x_3)|, |A^2_\sigma(x_3x_4)|, |A^2_\sigma(x_1x_4)| \ge 3$. So we must have equality and each of these sets must be the exact same, else we are done by SDR.  So without loss of generality, we may assume they are $\{2_5,2_6,2_7\}$, and further $U_\sigma(y_1) = U_\sigma(y_3) =  \{2_2, 2_3\}$ and $\{2_1,2_4\} \subseteq U_\sigma(y_2)$, and $\{2_1,2_4\} \subseteq U_\sigma(y_4)$. However this cannot happen as $y_2y_4$ was an edge in $G'$.


			\textbf{Case 4:}
			Suppose that we can color only one of the $x_iy_i$'s with 1.  We do so and call this good partial coloring $\tau$.    Without loss of generality, either only $x_2y_2$ or only $x_3y_3$ can be colored with 1.  So as a result $y_1$ is a 2-vertex in all the following, and $y_4$ is either a 2-vertex or $1 \in U_\phi(y_4)$.
			
			\textbf{Case 4.1:} Only $x_2y_2$ can be colored with 1, which implies that $y_3$ is also a 2-vertex.  Note that $y_1y_3$ cannot be an edge, else $x_2y_2$ and $x_4y_4$ would form an edge-cut of size two, so by Lemma \ref{lem:cut} $G$ would only have 7 vertices and we can easily provide a good coloring of $G$ by Lemma~\ref{nine}. Note that $|A^2_\tau(x_1y_1)|, |A^2_\tau(x_3y_3)| \ge 4$.  Since we have only seven 2-colors available, there must be a 2-color available on both $x_1y_1$ and $x_3y_3$, say $2_1$.  So we can further extend $\tau$ by coloring these edges with $2_1$, and continue to refer to this good partial coloring as $\tau$.  Note that $|A^2_\tau(x_1x_2)|, |A^2_\tau(x_2x_3)| \ge 3$, $|A^2_\tau(x_3x_4)|, |A^2_\tau(x_1x_4)| \ge 4$, and $|A^2_\tau(x_4y_4)| \ge 1$.  So the union of all five of these sets must have cardinality at most 4, else we are done by SDR.  In particular, we may assume $A^2_\tau(x_3x_4) = A^2_\tau(x_1x_4) = \{2_4, 2_5, 2_6, 2_7\}$  so that without loss of generality $U_\tau(y_1) = \{2_1,2_2\}$, and $2_3 \in U_\tau(y_4)$.  However since $y_2y_4$ exists in $G'$, $2_3 \notin U_\tau(y_2)$, which implies $2_3 \in A^2_\tau(x_1x_2)$ and $|A^2_\tau(x_1x_4) \cup A^2_\tau(x_1x_2)| \ge 5$, and we are done by SDR.

			\textbf{Case 4.2:} Only $x_3y_3$ can be colored with 1, so that $y_2$ is either a 2-vertex or $1 \in U_\phi(y_2)$.  Since $x_2y_2$ and $x_4y_4$ cannot be colored with 1, then $\phi(y_2y_4)$ was a 2-color, say $2_1$.  So we can further extend $\tau$ by coloring $x_2y_2$ and $x_4y_4$ with $2_1$.  In this good partial coloring of $G$ which we will continue to call $\tau$, $|A^2_\tau(x_1x_2)|,|A^2_\tau(x_4x_1)| \ge 4$,  and $|A^2_\tau(x_2x_3)|$, $|A^2_\tau(x_3x_4)|, |A^2_\tau(x_1y_1)| \ge 3$.  So we may assume without loss of generality that $A^2_\tau(x_1x_2) = A^2_\tau(x_4x_1) = \{2_4, 2_5, 2_6, 2_7\}$, else we are done by SDR.  However this implies without loss of generality $U_\tau(y_1) = \{2_2\}, \{2_1, 2_3\} \subseteq U_\tau(y_2)$ and $\{2_1, 2_3\} \subseteq U_\tau(y_4)$.  However this cannot happen as $y_2y_4$ was added in $G'$.

			\textbf{Case 5:}
			Suppose we cannot color any $x_iy_i$ with 1.  Then $y_1$ and $y_3$ must each be 2-vertices, for each $i \in \{2,4\}$, either $y_i$ is a 2-vertex or $1 \in U_\phi(y_i)$.  Note that if $y_1y_3$ was an edge, then $x_2y_2$ and $x_4y_4$ would form an edge-cut of size two, so by Lemma \ref{lem:cut} $G$ would only have 7 vertices and we can easily provide a good coloring of $G$ by Lemma~\ref{nine}.  Further, because we cannot color $x_2y_4$ or $x_4y_4$ with 1, $\phi(y_2y_4)$ was a 2-color, say $2_1$.  So we can extend the coloring $\phi$ onto $G$ by coloring $x_2y_2$ and $x_4y_4$ with $2_1$, and color $x_1x_2$ and $x_3x_4$ with 1.  In this good partial coloring of $G$, $|A^2(x_2x_3)|, |A^2(x_1x_4)| \ge 4$ and $|A^2(x_1y_1)|, |A^2(x_3y_3)| \ge 3$, so we are done by SDR.
		\end{proof}

		
		
		
		
		
		
		\begin{lemma}\label{distance4}
			A pair of $2$-vertices are at distance at least $4$.
		\end{lemma}
		
		\begin{proof}
			We first show there cannot be adjacent $2$-vertices. Let $u,v$ be adjacent $2$-vertices. Let $N(u) = \{u_1,v\}$ and $N(v) = \{u, v_1\}$. We delete $u,v$ from $G$ to obtain $G'$ and a good coloring $f$ on $G'$. If $u_1 = v_1$, then $uu_1$ and $vu_1$ see at most three colors respectively, and $uv$ see one color. Therefore, we can extend $f$ to a good coloring on $G$ by SDR.
			
			We then show $2$-vertices cannot be at distance two. Suppose not, $u_1,u_2$ are $2$-vertices at distance two. Let $N(u_1) = \{v_1, u_3\}$, $N(u_2) = \{v_2, u_3\}$. We know $u_3, v_1, v_2$ have degree three by Step 1. Let $N(u_3) = \{u_1, u_2, v_3\}$. By Lemma~\ref{lem:noK3}, there are no triangles and thus $v_3 \neq v_2$, $v_3 \neq v_1$, $v_1 \neq u_2$, and $v_2 \neq u_1$. Moreover, $v_1 \neq v_2$ by Lemma. 
			
			\textbf{Case 1:} $v_1v_2 \in E(G)$. We delete $u_1, u_2, u_3$ from $G$ to obtain $G'$ and a good coloring $f$ on $G'$. Since $v_1, v_2$ are both degree three vertices, we can recolor $v_1v_2$ with $1$. 
			
			\textbf{Case 1.1:} $v_3$ has degree two. Then $u_1v_1, u_2v_2, u_1u_3, u_2u_3, u_3v_3$ see at most $4,4,2,2,3$ colors and thus have at least $3,3,5,5,4$ available colors respectively. Therefore, we can extend $f$ to a good coloring on $G$ by SDR.
			
			\textbf{Case 1.2:} $v_3$ has degree three. Then we know $v_3$ was a degree two vertex in $G'$ and thus can color $v_3u_3$ with $1$. Moreover, $u_1v_1, u_2v_2, u_1u_3, u_2u_3$ see at most $4,4,3,3$ colors and thus have at least $3,3,4,4$ available colors respectively. Therefore, we can extend $f$ to a good coloring on $G$ by SDR.
			
			\textbf{Case 2:} $v_1v_2 \notin E(G)$. Then we delete $u_1, u_2, u_3$ from $G$ and add the edge $v_1v_2$ to obtain $G'$ and a good coloring $f$ on $G'$.
			
			\textbf{Case 2.1:} $f(v_1v_2) \in \{2_1, 2_2, 2_3, 2_4, 2_5, 2_6, 2_7\}$. Say $f(v_1v_2) = 2_1$. Then since $u_1v_1$ and $u_2v_2$ are at distance at least three, we color $u_1v_1$ and $u_2v_2$ with $2_1$. If $v_3$ has degree two, then $u_1u_3, u_2u_3, u_3v_3$ has at least $3,3,3$ colors available respectively and thus we can extend $f$ to a good coloring on $G$ by SDR; if $v_3$ has degree three, then it was a degree two vertex in $G'$ and thus we can color $u_3v_3$ with $1$, and since $u_1u_3, u_2u_3$ each has at least two available colors, we can extend $f$ to a good coloring on $G$ by SDR.
			
			\textbf{Case 2.2:} $f(v_1v_2) = 1$. Let $N(v_1) = \{u_1, v_3, v_4\}$ and $N(v_2) = \{u_2, v_5, v_6\}$. Then $u_1v_1$ and $u_2v_2$ have at least one color available respectively.
			
			\textbf{Case 2.2.1:} $|A^2_f(u_1v_1) \cup A^2_f(u_2v_2)| = 1$. Say $A^2_f(u_1v_1) \cup A^2_f(u_2v_2) = \{2_1\}$. Then we color $u_1v_1$ and $u_2v_2$ with $2_1$. If $v_3$ has degree two, then $u_1u_3, u_2u_3, u_3v_3$ has at least $3,3,3$ available colors respectively and thus we can extend $f$ to a good coloring on $G$ by SDR; if $v_3$ has degree three, then it was a degree two vertex in $G'$ and thus we can color $u_3v_3$ with $1$, and since $u_1u_3, u_2u_3$ each has at least two available colors, we can extend $f$ to a good coloring on $G$ by SDR.
			
			\textbf{Case 2.2.2:} $|A^2_f(u_1v_1) \cup A^2_f(u_2v_2)| \ge 2$. Say $2_1 \in A^2_f(u_1v_1)$ and $2_2 \in A^2_f(u_2v_2)$. Then we color $u_1v_1$ with $2_1$ and $u_2v_2$ with $2_2$. 
			
			If $v_3$ has degree two, then $u_1u_3, u_2u_3, u_3v_3$ has at least $2,2,2$ available colors respectively. By SDR, the only case for which we cannot extend $f$ to a good coloring on $G$ is when $A^2_f(u_3v_3) = A^2_f(u_1u_3) = A^2_f(u_2u_3)$ and has exactly two available colors. It is impossible as $|\{f(v_1v_3), f(v_1v_4), f(v_2v_5), f(v_2v_6)\}| = 4$, which implies $|A^2_f(u_1u_3) \cup A^2_f(u_2u_3)| \ge 3$.
			
			If $v_3$ has degree three, then it was a degree two vertex in $G'$ and thus we can color $u_3v_3$ with $1$, and since $|\{f(v_1v_3), f(v_1v_4), f(v_2v_5), f(v_2v_6)\}| = 4$, $|A^2_f(u_1u_3) \cup A^2_f(u_2u_3)| \ge 2$ and we can extend $f$ to a good coloring on $G$ by SDR.
			
			Next, we show $2$-vertices cannot be at distance three. Suppose not, let $u_1,u_2$ be two $2$-vertices at distance three. Let $u_1u_3u_5u_2$ be a path of length three connecting $u_1$ and $u_2$. Let $N(u_1) = \{u_3, u_4\}$, $N(u_3) = \{u_1, u_5, u_7\}$, $N(u_5) = \{u_2,u_3,u_8\}$, and $N(u_2) = \{u_5, u_6\}$. By Lemma~\ref{lem:noK3}, there are no triangles, we know that $u_4 \neq u_7$, $u_7 \neq u_8$, $u_6 \neq u_8$. By Lemma~\ref{4cycle}, there are no four cycles and thus $u_4 \neq u_8$ and $u_6 \neq u_7$. Moreover, by Step 2 we know $u_4 \neq u_6$, $u_4, u_6, u_7, u_8$ are $3$-vertices.
			
			\textbf{Case A:} $u_4u_6 \in E(G)$. Let $N(u_4) = \{u_1, u_6, u_9\}$ and $N(u_6) = \{u_2, u_4, u_{11}\}$. We delete $u_1, u_2, u_3, u_5$ from $G$ to obtain $G'$ and a good coloring $f$ on $G'$. Since $u_4, u_6, u_7, u_8$ are $2$-vertices in $G'$, we know $\{f(u_4u_9), f(u_4u_6), f(u_6u_{11})\} \subseteq \{2_1, 2_2, 2_3, 2_4, 2_5, 2_6, 2_7\}$. Thus, we can recolor $u_4u_6$ with $1$, color $u_3u_7$ and $u_5u_8$ with $1$. Say $f(u_4u_9) = 2_1$ and $f(u_6u_{11}) = 2_2$. Then $u_1u_4, u_1u_3, u_3u_5, u_2u_5, u_2u_6$ has at least $3,4,3,4,3$ available colors respectively and the only case for which $f$ cannot be extended to a good coloring on $G$ is when $|A^2_f(u_1u_4) \cup A^2_f(u_1u_3) \cup A^2_f(u_3u_5) \cup A^2_f(u_2u_5) \cup  A^2_f(u_2u_6)| = 4$ by SDR. Then we color $u_2u_5$ and $u_1u_4$ by a color $x \in A^2_f(u_2u_5) \cap A^2_f(u_1u_4)$ and $u_1u_3, u_3u_5,  u_2u_6$ has at least $3,2,2$ available colors respectively. Thus, we can extend $f$ to a good coloring on $G$ by SDR.
			
			\textbf{Case B:} $u_4u_6 \notin E(G)$. We delete $u_1, u_2, u_3, u_5$ from $G$ and add the edge $u_4u_6$ to obtain $G'$ and a good coloring $f$ on $G'$. Since $u_7, u_8$ are $2$-vertices in $G'$, we can color $u_3u_7$ and $u_5u_8$ with $1$. Then $u_1u_4, u_1u_3, u_3u_5, u_2u_5, u_2u_6$ have at least $1,3,3,3,1$ colors available respectively. Let $N(u_7) = \{u_3, u_{13}, u_{14}\}$ and $N(u_8) = \{u_5, u_{15}, u_{16}\}$.
			
			\textbf{Case B.1:} $A^2_f(u_1u_4) \cap A^2_f(u_2u_6) \neq \emptyset$. Say $2_1 \in A^2_f(u_1u_4) \cap A^2_f(u_2u_6)$. Then we color $u_1u_4$ and $u_2u_6$ with $2_1$. Now, $u_1u_3, u_3u_5, u_2u_5$ have at least $2,2,2$ colors available respectively; by SDR, we must have $|A^2_f(u_1u_3)| = |A^2_f(u_3u_5)| = |A^2_f(u_2u_5)| = 2$, say $2_2, 2_3$ are those two available colors and thus we may assume $u_4u_9 = u_8u_{15} = 2_4$, $u_4u_{10} = u_8u_{16} = 2_5$, $u_6u_{11} = u_7u_{13} = 2_6$, $u_6u_{12} = u_7u_{14} = 2_7$. 
			
			If $u_9$ does not have an adjacent edge colored by $1$, then we can recolor $u_4u_9$ with $1$ and we are done by SDR, since $2_4 \in A^2_f(u_1u_3)$ and $|A^2_f(u_1u_3)| \ge 3$. Thus, we may assume by symmetry $u_9, u_{10}, u_{11}, u_{12}$ do not have adjacent edge colored by $1$ and remove the colors in $u_1u_4$ and $u_2u_6$. We know $u_1u_4, u_1u_3, u_3u_5,$ $u_2u_5, u_2u_6$ have at least $3,3,3,3,3$ colors available. 
			
			\textbf{Case B.1.1:} $A^2_f(u_1u_4) \cap A^2_f(u_2u_5) = \emptyset$. Then $|A^2_f(u_1u_4) \cap A^2_f(u_2u_6)| \ge 6$. If $A^2_f(u_1u_3) \cap A^2_f(u_2u_6) = \emptyset$, then we are done by SDR; if $A^2_f(u_1u_3) \cap A^2_f(u_2u_6) \neq \emptyset$, then we color $u_1u_3$ and $u_2u_6$ with a color $x \in A^2_f(u_1u_3) \cap A^2_f(u_2u_6)$ and $u_1u_4, u_3u_5, u_2u_5$ have at least $2,2,2$ colors available respectively. Since $A^2_f(u_1u_4) \cap A^2_f(u_2u_5) = \emptyset$, we are done by SDR.
			
			\textbf{Case B.1.2:} $A^2_f(u_1u_4) \cap A^2_f(u_2u_5) \neq \emptyset$. Then we color $u_2u_5$ and $u_1u_4$ with a color $x \in A^2_f(u_1u_4) \cap A^2_f(u_2u_5)$. Now, $u_1u_3, u_3u_5, u_2u_6$ have at least $2,2,2$ colors available respectively. If $A^2_f(u_1u_3) \cap A^2_f(u_2u_6) \neq \emptyset$, then we color them with a color $y \in A^2_f(u_1u_3) \cap A^2_f(u_2u_6)$ and there is at least one colors left for $u_3u_5$; if $A^2_f(u_1u_3) \cap A^2_f(u_2u_6) = \emptyset$, then $|A^2_f(u_1u_3) \cap A^2_f(u_2u_6)| \ge 4$ and thus we are done by SDR.
			
			\textbf{Case B.2:} $A^2_f(u_1u_4) \cap A^2_f(u_2u_6) = \emptyset$. Then $u_1u_4$ and $u_2u_6$ each must have exactly one color available. To see this, we assume $|A^2_f(u_1u_4)| \ge 2$ and there are two subcases. (a) $A^2_f(u_2u_5) \cap A^2_f(u_1u_4) = \emptyset$. If $A^2_f(u_2u_6) \cap A^2_f(u_1u_3) \neq \emptyset$, then we color $u_2u_6$ and $u_1u_3$ with a color $y \in A^2_f(u_2u_6) \cap A^2_f(u_1u_3)$. Note that then $u_1u_4, u_3u_5, u_2u_5$ have at least $2,2,2$ colors available respectively and we are done by SDR since $|A^2_f(u_2u_5) \cup A^2_f(u_1u_4)| \ge 4$. If $A^2_f(u_2u_6) \cap A^2_f(u_1u_3) = \emptyset$, then $|A^2_f(u_2u_6) \cap A^2_f(u_1u_3)| \ge 4$ and $A^2_f(u_2u_5) \cap A^2_f(u_1u_4) \ge 5|$, and we are done by SDR. (b) $A^2_f(u_2u_5) \cap A^2_f(u_1u_4) \neq \emptyset$.  Then we color $u_2u_5$ and $u_1u_4$ with a color $x \in A^2_f(u_1u_4) \cup A^2_f(u_2u_5)$. Each of $u_1u_3, u_3u_5, u_2u_6$ has at least $2,2,1$ colors available respectively. If $A^2_f(u_1u_3)$ and $A^2_f(u_2u_6)$ has a color in common, then use it and there is at least one color left for $u_3u_5$. If not, then $u_1u_3, u_3u_5, u_2u_6$ have in total at least three available colors and we are done by SDR.
			
			Then we color $u_1u_4$ with $x$ and $u_2u_6$ with $y$ such that $x \neq y$. We must also have each of $u_1u_4, u_2u_6$ see exactly $6$ colors. Thus, we can recolor one edge adjacent to each of $u_4,u_6$ with $1$ and now $u_1u_3, u_3u_5, u_2u_5$ have $3,1,3$ colors available respectively. We are done by SDR.
		\end{proof}

		\begin{lemma}\label{atmostone}
			There is at most one $2$-vertex in $G$.
		\end{lemma}
		
		\begin{proof}
			Let $N \ge 3$ be an integer. We show that there are no pair of $2$-vertices at distance $N$. We prove this by induction. The base case was done in Step 3 of Lemma~\ref{distance4}. We assume now there are no pair of $2$-vertices at distance at most $N-1$.
			
			Suppose to the contrary that $u_1, u_2$ are $2$-vertices at distance $N$. Let $u_1v_1 \ldots v_{N-1}u_2$ be a path of length $N$ connecting $u_1$ and $u_2$. Let $N(u_1) = \{v_1, u_3\}$, $N(v_1) = \{u_1, w_1, v_2\}$, $\ldots$, $N(v_i) = \{v_{i-1}, w_2, v_{i+1}\}$, $\ldots$, $N(v_{N-1}) = \{v_{N-2}, w_{N-1}, u_2\}$, $N(u_2) = \{v_{N-1}, u_4\}$, $N(u_3) = \{u_1, u_5, u_6\}$, $N(u_4) = \{u_2, u_7, u_8\}$. By Step~1-3 in Lemma~\ref{distance4} and the induction hypothesis, we know $u_3, u_5, u_6, u_4, u_7, u_8$ are $3$-vertices. By Lemma~\ref{lem:noK3}, $u_3 \neq w_1$, $w_i \neq w_{i+1}$ for $i \in [N-2]$, and $w_{N-1} \neq u_2$. By Lemma~\ref{4cycle} and there are no $2$-vertices at distance at most $N-1$, $|\{u_1, u_2, u_3, u_4, w_1, \ldots, w_{N-1}, v_1, \ldots, v_{N-1}\}| = 2N+2$; for the same reason, $\{u_5, u_6\} \cap \{w_1, \ldots, w_{N-1}\} \subseteq \{w_2, w_3\}$, $\{u_7, u_8\} \cap \{w_1, \ldots, w_{N-1}\} \subseteq \{w_{N-3}, w_{N-2}\}$.
			
			We know $u_3u_4 \notin E(G)$ by Lemma~\ref{distance4}. We delete $u_1, u_2, v_1, \ldots, v_{N-1}$ from $G$ to obtain a graph $G'$ and a good coloring $f$ on $G'$. Since each of $\{w_1, \ldots, w_{N-1}\}$ is a $2$-vertex in $G'$, we can color each $w_iv_i$ with $1$. Then $u_1u_3, u_1v_1, v_1v_2, \ldots, v_{N-2}v_{N-1}, v_{N-1}u_2$ have at least $1, 3, \ldots, 3, 1$ colors available respectively. Moreover, we have $u_1u_3$ and $u_1v_1$ either have at least $2,4$ or $3,3$ colors available. To see this, if one of $u_3u_5$ and $u_3u_6$ are colored with $1$, then $u_1u_3$ and $u_1v_1$ have at least $2,4$ colors available respectively; if $u_3u_5$ and $u_3u_6$ are both not colored with $1$, and we cannot recolor one of them with $1$, then each of them is adjacent to an edge colored with $1$ and thus $u_1u_3, u_1v_1$ have $3,3$ colors available respectively. The same conclusion holds for $u_2u_4, u_2v_{N-1}$. Then $u_1u_3, u_1v_1, v_1v_2, \ldots, v_{N-3}v_{N-2}, v_{N-2}v_{N-1}, v_{N-1}u_2$ have $x_1, x_2, 3, \ldots, 3, x_3, x_4$ colors available respectively, where $x_1, x_2, x_3, x_4 = (2,4,4,2), (2,4,3,3), (3,3,4,2), \text{ or }(3,3,3,3)$.
			
			\textbf{Case 1:} $(2,4,4,2)$. Then we start by coloring $u_1u_3$ with a color $y_0 \in A^2_f(u_1u_3)$, $u_1v_1$ with a color $y_1 \in A^2_f(u_1v_1) - y_0$, $v_1v_2$ with a color $y_2 \in A^2_f(v_1v_2) - \{y_0, y_1\}$, $\ldots$, $v_iv_{i+1}$ with a color $y_{i+1} \in A^2_f(v_iv_{i+1}) - \{y_{i-1}, y_i\}$, where $1 \le i \le N-2$. We can always do that since each of $v_iv_{i+1}$ has at least three colors available, where $i \in [N-2]$. Then we color $u_2u_4$ with a color $y \in A^2_f(u_2u_4) - y_{N-1}$ and color $u_2v_{N-1}$ with a color $z \in A^2_f(u_2v_{N-1}) - \{y, y_{N-2}, y_{N-1}\}$.
			
			\textbf{Case 2:} $(2,4,3,3)$. Similarly to the procedure described in Case 1, we can start coloring from $u_1u_3$ to $u_2u_4$ (proceed the main procedure in Case 1 through $u_2v_{N-1}, u_2u_4$  without the last sentence in Case 1).
			
			\textbf{Case 3:} $(3,3,4,2)$. Similarly to Case 2, we can start coloring from $u_2u_4$ to $u_1u_3$.
			
			\textbf{Case 4:} $(3,3,3,3)$. Similarly to Case 1, we can start coloring from $u_1u_3$ to $u_2u_4$ (proceed the main procedure in Case 1 through $u_2v_{N-1}, u_2u_4$  without the last sentence in Case 1).
		\end{proof}
		
		\begin{lemma}\label{no2vx}
			There is no $2$-vertex.
		\end{lemma}
		
		\begin{proof}
			By Lemma~\ref{atmostone}, we know that there is at most one $2$-vertex, say $u$. We claim that $u$ is contained in some cycle $C$ since otherwise $u$ is a cut-vertex, which is a contradiction to Lemma~\ref{lem:cutedge}. By Lemma~\ref{4cycle}, $C$ has at least five vertices. Let $C = u, u_1, u_2, \ldots, u_k, u$, where $k \ge 4$. Since $u$ is the only $2$-vertex, we can let the neighbours of $u_1, u_2, \ldots, u_k$ that is outside of $C$ be $v_1, v_2, \ldots, v_k$ respectively. 
			
			We delete $V(C)$ from $G$ to obtain a graph $G'$ and a good coloring $f$ on $G'$. Since $u$ is the only $2$-vertex in $G$, we know that $v_1, \ldots, v_k$ are $2$-vertices in $G'$. Thus, we can color the edges $u_1v_1, \ldots, u_kv_k$ by $1$. Furthermore, we know $|A^2_f(uu_1)|, |A^2_f(u_1u_2)|, \ldots, |A^2_f(u_{k-1}u_{k})|, |A^2_f(u_ku)|$ are at least $5, 3, \ldots, 3, 5$ respectively. Therefore, we can start coloring the edge $u_1u_2$ by any color in $A^2_f(u_1u_2)$, then color the edge $u_2u_3$ by a color in $A^2_f(u_2u_3)-f(u_1u_2)$, the edge $u_3u_4$ by a color in $A^2_f(u_3u_4)-f(u_2u_3)-f(u_1u_2)$, $\ldots$, color the edge $u_{k-1}u_k$ by a color in $A^2_f(u_{k-1}u_k)-f(u_{k-2}u_{k-1})-f(u_{k-3}u_{k-2})$; the process can be done since at most two colors in $A^2_f(u_iu_{i+1})$ are already used by the previously colored neighbours of $u_iu_{i+1}$ within distance two in $E(C)$ when we color each edge $u_iu_{i+1}$. Then we color $uu_k$ by a color in $A^2_f(uu_k)-f(u_{k-1}u_k)-f(u_{k-2}u_{k-1})-f(u_1u_2)$ and then we color $uu_1$ by a color in $A^2_f(uu_1)-f(uu_k)-f(u_{k-1}u_k)-f(u_1u_2)-f(u_2u_3)$.
			
			This extends the good coloring $f$ on $G'$ to a good coloring on $G$.
		\end{proof}
		
		

		\begin{lemma}\label{5cycle}
			There are no 5-cycles.
		\end{lemma}
		
		\begin{proof}
			Suppose not and let $C=x_1, x_2, x_3, x_4, x_5, x_1$ be a 5-cycle. By Lemma~\ref{no2vx}, all vertices have degree three. Let $y_i$ be the neighbor of $x_i$ not in $C$. By Lemma~\ref{lem:noK3} and~\ref{4cycle}, there are no 3-cycles and 4-cycles in $G$, and thus we know $y_i\neq y_j$ for each $i\neq j$ and $y_iy_{i+1 (mod\ 5)} \notin E(G)$ for $i \in [5]$.
			
			\textbf{Case 1:} $y_iy_{i+2} \in E(G)$ for all $i,i+2 \pmod 5$.

			The resulting graph is the Petersen graph which admits a $(1,2^5)$-packing edge-coloring. Since the Petersen graph is cubic, the $(1,2^5)$-packing edge-coloring is also a good coloring.

			\textbf{Case 2:} One of the edge $y_iy_{i+2} \notin E(G)$ for some $i,i+2 \pmod 5$. Say $y_2y_5 \not \in E(G)$.

			Consider the graph $G'$ obtained by deleting the vertices $x_1,x_2,x_3,x_4,x_5$ and adding the edge $y_2y_5$. Let $f$ be a good coloring of $G'$. We now extend $f$ to a coloring of $G$ by considering the following cases.
			
			\textbf{Case 2.1:} $1 \not \in A(x_2y_2)$ but $1 \in A(x_5y_5)$.
			
			For $i\in \{1,3,4\}$, since $y_i$ has degree 2 in $G'$, $1\in A(x_iy_i)$. Then we color the edges $x_iy_i$ with 1 for $i\in \{1,3,4,5\}$. The edges $x_2y_2, x_2x_3, x_3x_4, x_4x_5, x_5x_1, x_1x_2$ see at most $5,3,4,4,4,3$ colors and so they have at least $2,4,3,3,3,4$ colors available, respectively. Note that since $y_2$ and $y_5$ are adjacent in $G'$, $|A^2(x_5x_1)\cup A^2(x_1x_2)|\ge 5$.
			
			\textbf{Case 2.1.1:} $A^2(x_2y_2)\cap A^2(x_4x_5) \neq \emptyset$.
			
			If $x_2y_2$ and $x_4x_5$ have an available color in common, then we use that color on both edges. Now, the edges $x_2x_3, x_3x_4, x_5x_1, x_1x_2$ have at least $3,2,2,3$ colors available, respectively with $|A^2(x_5x_1\cup A^2(x_1x_2)|\ge 4$. Thus we are done by SDR.
			
			\textbf{Case 2.1.2:} $A^2(x_2y_2)\cap A^2(x_4x_5) = \emptyset$.
			
			Since $x_2y_2$ and $x_4x_5$ have at least 2 and 3 colors available, respectively, and  $A^2(x_2y_2)\cap A^2(x_4x_5) = \emptyset$, then we must have $|A^2(x_2y_2)\cup A^2(x_4x_5)|\ge 5$. Then along with the counts of available colors established above, shows that if the union of the available colors for the edges $x_2y_2, x_2x_3, x_3x_4, x_4x_5, x_5x_1, x_1x_2$ contains at least 6 colors, then we are done by SDR. So let us assume that $x_2y_2$ and $x_4x_5$ have exactly 2 and 3 colors available. Without loss of generality, assume that $U(y_5)=\{2_1, 2_2\}$ and $U(y_4)=\{2_3,2_4\}$. Then we must have $A^2(x_4x_5)=\{2_5, 2_6, 2_7\}$ and $2_1, 2_2 \notin U(y_2)$.
			
			Since $f$ was a good coloring of $G'$ and $A^2(x_2y_2)\cap A^2(x_4x_5) = \emptyset$, then  the edge $y_2y_5$ must have received color $2_3$ or $2_4$ in $G'$ and hence, $2_3\in A^2(x_2y_2)$ or $2_4\in A^2(x_2y_2)$. Without loss of generality assume that $2_3\in A^2(x_2y_2)$. Also note that since $f$ was a good coloring of $G'$ and $U(y_5)=\{2_1, 2_2\}$, then $2_1,2_2 \not \in U(y_2)$. So if $2_1,2_2 \not \in U(y_1)$, then $2_1,2_2 \in A^2(x_1x_2)$ and thus, $|A^2(x_1x_2)\cup A^2(x_2y_2) \cup A^2(x_4x_5)| \ge 6$ and we are done by SDR. Thus we must have $2_1\in U(y_1)$ or $2_2\in U(y_1)$. Similarly, if $2_1,2_2 \not \in U(y_3)$, then $2_1,2_2 \in A^2(x_2x_3)$ and we are done by SDR. So we must also have $2_1\in U(y_3)$ or $2_2\in U(y_3)$. Furthermore, we must have $2_1\in U(y_1)\cap U(y_3)$ or $2_2\in U(y_1)\cap U(y_3)$, else one of the edges $x_1x_2, x_2x_3$ would have $2_1$ available and the other would have $2_2$ available and we would be done by SDR. Without loss of generality, assume that $2_1\in U(y_1)\cap U(y_3)$. Note that this implies the edge $x_5x_1$ has at least 4 colors available. Let $\alpha$ be the color used at $y_1$ that is not $2_1$. We proceed by considering cases on $\alpha$.
			
			First, if $\alpha \in \{2_3, 2_5, 2_6, 2_7\}$, then $2_4\in A^2(x_1x_5)$ and $2_2\in A^2(x_1x_2)$. Then the edges $x_1x_2$, $x_2y_2$, $x_2x_3$, $x_3x_4$, $x_4x_5$, $x_5x_1$ have at least 4, 2, 4, 3, 3, 4 colors available, respectively, with $|A^2(x_2y_2)\cup A^2(x_4x_5)|\ge 5$, $|A^2(x_1x_5)\cup A^2(x_1x_2)|\ge 5$, and $|A^2(x_1x_5)\cup A^2(x_1x_2)\cup A^2(x_2y_2)|\ge 6$, since $2_3 \in A^2(x_2y_2)$. Thus, we are done by SDR.
			
			Lastly, suppose that $\alpha \in \{2_2, 2_4\}$. Let $\alpha= 2_2 (2_4)$. Then we must have $2_1, 2_2 (2_4) \not \in A^2(x_2y_2)$, else we'd be done by SDR similar to the case above. This implies that $A^2(x_2y_2)=\{2_3, 2_4 (2_2)\}$. If $x_1y_1$ and $x_3x_4$ have an available color in common, then we uncolor $x_1y_1$, use that common color on both edges, and recolor $x_1x_2$ with 1. Now the edges $x_5x_1$, $x_2x_3$ have 4,3 colors available, respectively, and $|A^2(x_4x_5)\cup A^2(x_2y_2)|\ge 4$, so we are done by SDR. If $x_1y_1$ and $x_3x_4$ do not have an available color in common, then consider the number of available colors at $x_1y_1$. Let $u_1$ and $w_1$ be the vertices adjacent to $y_1$ not in $C$. If $|A^2(x_1y_1)|\ge 3$, then we uncolor $x_1y_1$ and color $x_1x_2$ with $1$. Since $|A^2(x_1y_1)\cup A^2(x_3x_4)|\ge 6$, we are done by SDR. Otherwise, we must have $|A^2(x_1y_1)| \le 2$, so we can recolor one of the edges $y_1u_1$ or $y_1w_1$ with 1, say $y_1u_1$, and color the edge $x_1y_1$ with color $f(y_1u_1)$. Note that $f(y_1u_1) \in \{2_1,2_2 (2_4)\}$, so that the edges $x_2y_2$, $x_2x_3$, $x_3x_4$, $x_4x_5$, $x_1x_5$ have at least $2,3,3,3,5(4)$ colors available, respectively, with $|A^2(x_2y_2)\cup A^2(x_4x_5)| \ge 5$, so we are done by SDR.

			
			
			
			\textbf{Case 2.2:} $1 \not \in A(x_2y_2)\cup A(x_5y_5)$.
			
			For $i\in \{1,3,4\}$, since $y_i$ has degree 2 in $G'$, $1\in A(x_iy_i)$. Then we color the edges $x_iy_i$ with 1 for $i\in \{3,4\}$ and leave the edge $x_1y_1$ uncolored for now. Without loss of generality, assume that $2_1\in U(y_5)$, $2_2\in U(y_2)$, and the edge $y_2y_5$ received color $2_3$ in $G$. Then we color the edges $x_2y_2, x_5y_5$ with $2_3$. The edges $x_1x_2, x_2x_3, x_4x_5, x_5x_1$ each see at most $4$ colors and so they each have at least $3$ colors available. Note that since $y_2$ and $y_5$ are adjacent in $G'$, $|A^2(x_5x_1) \cup A^2(x_1x_2)|\ge 4$. The edge $x_3x_4$ sees at most 5 colors, so there are at least 2 colors available on this edge. 
			
			\textbf{Case 2.2.1:} $A^2(x_2x_3)=A^2(x_4x_5)$ with $|A^2(x_2x_3)\cup A^2(x_4x_5)|=3$.
			
			Since $A^2(x_2x_3)=A^2(x_4x_5)$, $2_1\not \in A^2(x_4x_5)$, and $2_2\not \in A^2(x_2x_3)$, then we must have $2_1\in U(y_3)$ and $2_2\in U(y_4)$. Then since $A^2(x_2x_3)=A^2(x_4x_5)$ and $|A^2(x_2x_3)\cup A^2(x_4x_5)|=3$, without loss of generality assume that $2_4\in U(y_3)\cap U(y_4)$. Thus we must have $A^2(x_2x_3)=A^2(x_3x_4)=A^2(x_4x_5)=\{2_5, 2_6, 2_7\}$. 
			
			\textbf{Case 2.2.1.1:} If one of the colors $2_5, 2_6, 2_7$ is available on the edge $x_1y_1$, say $\beta$, then we use $\beta$ on the edges $x_1y_1, x_3x_4$. Note that since $|A^2(x_5x_1) \cup A^2(x_1x_2)|\ge 4$, then one of the edges $x_5x_1, x_1x_2$ has a color available that does not belong to $\{2_5, 2_6, 2_7\}$. Let this color be denoted $\alpha$ and without loss of generality, suppose $\alpha$ is available on the edge $x_5x_1$. Then we color the edge $x_1x_2$ with 1 and color the edge $x_5x_1$ with $\alpha$. Since $x_2x_3$ and $x_4x_5$ both still have at least 2 colors available, then we are done by SDR.  
			
			\textbf{Case 2.2.1.2:} Colors $2_5, 2_6, 2_7$ are not available on the edge $x_1y_1$ but we can color the edge $x_1y_1$ with $2_1$ or $2_2$. Then we can extend $f$ to a good coloring of $G$. By symmetry, suppose we can color $x_1y_1$ with $2_1$. Then the edge $x_5x_1$ still has 3 colors available and these three colors must be $2_5,2_6,2_7$; else we would color $x_1x_2$ with 1 and we are done by SDR. Let $u_1$ and $w_1$ be the vertices adjacent to $y_1$ not in $C$. Since $A^2(x_5x_1)=\{2_5,2_6,2_7\}$, then the edges $y_1u_1$ and $y_1w_1$ must be colored $2_2$ and $2_4$, so that $U(y_1)=\{2_1,2_2,2_4\}$. Now since the edge $x_1y_1$ also sees the colors $2_5, 2_6, 2_7$, this implies we can recolor one of the edges $y_1u_1, y_1w_1$ with 1, which makes a color $\beta \in \{2_2,2_4\}$ available at $x_5x_1$. Then we color $x_5x_1$ with $\beta$ and $x_1x_2$ with 1. Since each of the edges $x_2x_3, x_3x_4, x_4x_5$ still has 3 colors available, we are done by SDR.
			
			\textbf{Case 2.2.1.3:} Colors $2_1,2_2,2_5,2_6,2_7$ are not available on the edge $x_1y_1$. Again let $u_1$ and $w_1$ be the vertices adjacent to $y_1$ not in $C$. Since $x_1y_1$ sees $2_1,2_2,2_5,2_6,2_7$, then at least one of the edges $y_1u_1,y_1w_1$ is colored with $\gamma \in \{2_1,2_2,2_5,2_6,2_7\}$ and the endpoint do not see color $1$, say $y_1u_1$ is colored with $\gamma$ and $u_1$ do not see color 1. 
			
			If $\gamma \in \{2_5,2_6,2_7\}$, then we color the edge $y_1u_1$ with 1, color the edge $x_1y_1$ with $\gamma$, and we are done as in \textbf{Case~2.2.1.1} above. Thus we can assume $\gamma \in \{2_1,2_2\}$. By symmetry, suppose $\gamma = 2_1$. Then we color the edge $y_1u_1$ with 1 and color the edge $x_1y_1$ with $2_1$. Now the edge $x_5x_1$ has at least 4 colors available. We color the edge $x_1x_2$ with 1. Since the edges $x_5x_1, x_2x_3, x_3x_4, x_4x_5$ have at least 4,3,3,3 colors available, respectively, then we are done by SDR. 
			
			\textbf{Case 2.2.2:} $|A^2(x_2x_3)\cup A^2(x_4x_5)|\ge 4$.
			
			We begin this case by coloring $x_1y_1$ with 1. Note that since $|A^2(x_2x_3)\cup A^2(x_4x_5)|\ge 4$ and $|A^2(x_5x_1\cup A^2(x_1x_2)|\ge 4$, then we are done by SDR unless the union of available colors on the edges of $C$ is at most 4. So let us assume it is exactly 4. We consider cases based on whether $2_1$ and or $2_2$ are available colors on the edges of $C$.
			
			\textbf{Case 2.2.2.1:} Both $2_1$ and $2_2$ are available colors. Without loss of generality, suppose the union of available colors on the edges of $C$ are $\{2_1,2_2,2_6,2_7\}$. Then for $i\in \{1,3,4\}$, we must have $2_4,2_5\in U(y_i)$ and hence, the edge $x_3x_4$ must have all 4 colors $2_1,2_2,2_6,2_7$ available. We proceed by removing the color 1 from the edge $x_1y_1$.
			
			If $x_1y_1$ and $x_3x_4$ have a color $\alpha$ in common, then we use that color on both of these edges. If $\alpha = 2_1$, then we color the edge $x_1x_2$ with 1, and we are done by SDR since the edges $x_5x_1, x_2x_3, x_4x_5$ have at least 3,2,3 colors available, respectively. Similarly, if $\alpha = 2_2$, then we color the edge $x_5x_1$ with 1, and we are done by SDR. If $\alpha \in \{2_6,2_7\}$, then we color $x_1x_2$ with $2_1$, color $x_5x_1$ with 1, and since we have used two of the available colors, we must still have $|A^2(x_2x_3)\cup A^2(x_4x_5)|\ge 2$ so we are done by SDR.
			
			Now we consider the case where $x_1y_1$ and $x_3x_4$ have no color in common. Let $u_1$ and $w_1$ be the vertices adjacent to $y_1$ not in $C$. Since $2_4,2_5\in U(y_1)$, without loss of generality assume $y_1w_1$ is colored $2_4$ and $y_1u_1$ is colored with $2_5$. Since $x_1y_1$ and $x_3x_4$ have no color in common, then all the colors $2_1, 2_2, 2_6, 2_7$ are not available on $x_1y_1$, which means we can recolor the edge $y_1w_1$ with 1, color the edge $x_1y_1$ with $2_4$, and color the edge $x_1x_2$ with 1. Then the edges $x_2x_3, x_3x_4, x_4x_5, x_5x_1$ have at least $3,4,3,3$ colors available, respectively, so we are done by SDR.
			
			\textbf{Case 2.2.2.2:} Only one of $2_1$ or $2_2$ is an available color on the edges of $C$. By symmetry, suppose that $2_1$ is available and without loss of generality, assume the union of available colors on the edges of $C$ are $\{2_1,2_5.2_6,2_7\}$. Then we must have $2_2,2_4\in U(y_4)$ and $2_2,2_4\in U(y_1)$, which implies $A^2(x_4x_5)=\{2_5,2_6,2_7\}$ and $A^2(x_5x_1)=\{2_5,2_6,2_7\}$. We must also have $2_4\in U(y_3)$. Now since $A^2(x_4x_5)=\{2_5,2_6,2_7\}$ and $|A^2(x_2x_3)\cup A^2(x_4x_5)|\ge 4$, then we must have $2_1\in A^2(x_2x_3)$, so none of the edges incident to $y_3$ are colored with $2_1$. This implies that $2_1\in A^2(x_3x_4)$, so without loss of generality assume that $A^2(x_3x_4)=\{2_1,2_6,2_7\}$. We proceed again by removing the color 1 from the edge $x_1y_1$.
			
			Suppose $x_1y_1$ and $x_3x_4$ have an available color in common, say this color is $\alpha$. Then we use this color on both edges $x_1y_1$ and $x_3x_4$. If $\alpha \in \{2_6,2_7\}$, then we color the edge $x_2x_3$ with $2_1$ and the edge $x_1x_2$ with 1. The remaining edges $x_4x_5$ and $x_5x_1$ each have at least $2$ colors available, so we are done by SDR. If $\alpha = 2_1$, then we color the edge $x_1x_2$ with 1. The remaining edges $x_2x_3, x_4x_5, x_5x_1$ have at least $2,3,3$ colors available, respectively, so we again are done by SDR.
			
			Now suppose $x_1y_1$ and $x_3x_4$ have no available color in common. Let $u_1$ and $w_1$ be the vertices adjacent to $y_1$ not in $C$ and without loss of generality, assume $y_1u_1$ is colored $2_2$ and $y_1w_1$ is colored $2_4$. Since $x_1y_1$ and $x_3x_4$ have no available color in common, then $2_5,2_6,2_7 \in U(u_1)\cup U(w_1)$, which implies we can recolor one of the edges $y_1u_1$ or $y_1w_1$ with 1 and color the edge $x_1y_1$ with color $\alpha \in \{2_2,2_4\}$. Then we color the edge $x_5x_1$ with 1. Then the remaining edges $x_1x_2, x_2x_3, x_3x_4, x_4x_5$ have at least $4,3,3,3$ colors available, so we are done by SDR.
			
			\textbf{Case 2.2.2.3:} Neither $2_1$ nor $2_2$ are available colors on the edges of $C$. Then the union of available colors for the edges in $C$ must be $\{2_4,2_5,2_6,2_7\}$. Note that we must have $2_1,2_2\in U(y_1)$ so that $2_1, 2_2$ are not available on the edges $x_1x_2, x_5x_1$ and thus, $A^2(x_1x_2)=A^2(x_5x_1)=\{2_4, 2_5, 2_6, 2_7\}$. Additionally, we must have $2_1\in U(y_3)$ so that $2_1$ is not available on the edges $x_2x_3$ and $2_2\in U(y_4)$ so that $2_2$ is not available on the edge $x_4x_5$. We proceed by removing the color 1 from the edge $x_3y_3$. Let $u_3$ and $w_3$ be the vertices adjacent to $y_3$ not in $C$.
			
			If $x_5x_1$ and $x_3y_3$ have an available color in common, then we use that color on both edges and color $x_2x_3$ with 1. The remaining edges $x_1x_2, x_3x_4, x_4x_5$ have at least $3,1,2$ colors available, respectively, so we are done by SDR. 
			
			Otherwise, the colors $2_4, 2_5, 2_6, 2_7$ are not available on $x_3y_3$. Then we must have $U(u_3)\cup U(w_3)=\{2_1, 2_4, 2_5, 2_6, 2_7\}$, which implies we can recolor one of the edges $y_3u_3$ or $y_3w_3$ with 1. Suppose we can recolor $y_3u_3$ with 1 and let $\beta$ be the color assigned to the edge $y_3u_3$ under the coloring $f$. We proceed by coloring the edge $x_3y_3$ with $\beta$ and $x_2x_3$ with 1. Since the remaining edges $x_1x_2, x_3x_4, x_4x_5, x_5x_1$ have at least $3,2,2,4$ colors available, then we are done by SDR.
			
			\textbf{Case 2.3:} $1 \in A^2(y_2)\cap A^2(y_5)$.
			
			For $1\le i \le 5$, since $y_i$ has degree 2 in $G'$, $1\in A^2(y_i)$. Then we color the edges $x_iy_i$ with 1 for $1\le i \le 5$. Without loss of generality, assume that $2_1,2_2\in U(y_5)$ and $2_3,2_4\in U(y_2)$. The edges $x_1x_2, x_2x_3, x_3x_4, x_4x_5, x_5x_1$ each see at most $4$ colors and so they each have at least $3$ colors available. Note that since $y_2$ and $y_5$ are adjacent in $G'$, $|A^2(x_5x_1\cup A^2(x_1x_2)|\ge 5$. Hence we are done by SDR unless there are exactly 3 colors available on each of the edges $x_2x_3, x_3x_4, x_4x_5$, these colors are the same, and the same 3 colors are the only colors available on one of the edges $x_1x_2$ or $x_5x_1$. Suppose the same 3 colors are available on the edge $x_1x_2$. Since $A^2(x_2x_3)=A^2(x_3x_4)=A^2(x_4x_5)=A^2(x_1x_2)$, then we must have $2_1,2_2\in U(y_1)$, $2_1,2_2\in U(y_3)$, and $2_3,2_4\in U(y_4)$, which implies the available colors on the edges $x_1x_2, x_2x_3, x_3x_4, x_4x_5$ are $2_5,2_6,2_7$. We proceed by removing the color 1 from the edges $x_1y_1$ and $x_2y_2$ and coloring the edge $x_1x_2$ with 1. For $i=1,2$ let $u_i$ and $w_i$ be the vertices adjacent to $y_i$ not in $C$. 
			
			Suppose we can recolor one of the edges $y_1u_1$ or $y_1w_1$ with 1 and that we can recolor one of the edges $y_2u_2$ or $y_2w_2$ with 1, say we can recolor $y_1u_1$ and $y_2u_2$. Let $\alpha, \beta$ be the color $y_1u_1, y_2u_2$ received under $f$, respectively. Then we color the edge $x_1y_1$ with $\alpha$ and color the edge $y_2u_2$ with $\beta$. Then the remaining uncolored edges $x_2x_3, x_3x_4, x_4x_5, x_5x_1$ have at least 3,3,3,4 colors available respectively, so we are done by SDR.
			
			Hence, we must be unable to recolor both the edges $y_iu_i$ and $y_iw_i$ for at least one of $i=1$ or $i=2$. This implies that one of the edges $x_1y_1,x_2y_2$ has at least 3 colors available, while the other has at least 1 color available. We may assume that $x_2y_2$ has at least 3 colors available and the case when $x_1y_1$ has at least 3 colors available can be done by a similar argument (repeat the following argument word by word with the role of $x_1y_1$ and $x_2y_2$ switched). We proceed by considering cases on whether the colors $2_5,2_6,2_7$ are available on the edges $x_1y_1,x_2y_2$. By our assumption, $|A^2(x_2y_2) \cap \{2_5, 2_6, 2_7\}| \ge 1$.
			
			\textbf{Case 2.3.1:} The color $\alpha $ is available on the edge $x_1y_1$ and the color $\beta$ is available on the edge $x_2y_2$, where $\alpha \neq \beta$ and $\alpha, \beta \in \{2_5,2_6,2_7\}$. Then we proceed by coloring the edges $x_1y_1, x_3x_4$ with $\alpha$ and coloring the edges $x_2y_2,x_4x_5$ with $\beta$. The remaining uncolored edges $x_2x_3$ and $x_5x_1$ have at least 1,3 colors available, so we are done by SDR. 
			
			\textbf{Case 2.3.2:} Only one of the colors $2_5,2_6,2_7$ is available on both $x_1y_1$ and $x_2y_2$ and let this color be $\alpha$. If $\alpha$ is available on both edges $x_1y_1$ and $x_2y_2$, then consider the edge with fewer available colors, which is $x_1y_1$ by our previous assumption. Then we color $x_1y_1,x_3x_4$ with color $\alpha$. The edges $x_2y_2, x_2x_3, x_4x_5, x_5x_1$ have at least $2,2,2,4$ colors available and thus the only case we are not done is when $x_2y_2, x_2x_3, x_4x_5$ have the same two available colors, say $\gamma, \delta \in \{2_5, 2_6, 2_7\}$. Then we use $\gamma$ to color $x_2y_2$ and $x_4x_5$, color $x_2x_3$ with $\delta$, and there is at least one available color left to use at $x_1x_5$.

			\textbf{Case 2.3.3:} Only $x_2y_2$ has an available color in $\{2_5, 2_6, 2_7\}$, say $\alpha$. Since none of $2_5, 2_6, 2_7$ is available at $x_1x_2$ and $x_1x_2$ has at least one available color, we color $x_1x_2$ with a color in $A^2(x_1y_1)$, say $2_3$. Then we color $x_2y_2$ and $x_4x_5$ with $\alpha$. We find that now $|A^2(x_1x_5)| \ge 3$, $|A^2(x_2x_3)|, |A^2(x_3x_4)| \ge 2$, and thus we are done by SDR.		
		\end{proof}
		
		\begin{lemma}\label{6cycle}
			There are no $6$-cycles.
		\end{lemma}
		
		\begin{proof}
			By Lemma~\ref{no2vx}, every vertex is a $3$-vertex. Suppose there is a $6$-cycle $C=x_1, x_2, x_3, x_4, x_5, x_6, x_1$ and the neighbours of each $x_i$ outside of the cycle be $y_i$, where $1 \le i \le 6$. By Lemma~\ref{lem:noK3},~\ref{4cycle},~\ref{5cycle}, there are no cycles of length at most 5 and thus all $y_i$'s are distinct. We delete $\{x_1, x_2, \ldots, x_6\}$ from $G$ to obtain a graph $G'$ and a good coloring $f$ on $G'$. We extend $f$ to $G$.
			
			Since $f'$ is a good coloring on $G'$, we can color each edge $x_iy_i$ by color $1$. We know each $|A^2(x_ix_{i+1 (mod\ 6)})| \ge 3$, where $1 \le i \le 6$.
			
			\textbf{Case 1:} $A^2(x_1x_2) \cap A^2(x_4x_5) \neq \emptyset$. Say $2_1 \in A^2(x_1x_2) \cap A^2(x_4x_5)$. We color both $x_1x_2, x_4x_5$ by $2_1$. 
			
			\textbf{Case 1.1:} $A^2(x_1x_6) \cap A^2(x_3x_4) - 2_1 \neq \emptyset$. Say $2_2 \in A^2(x_1x_6) \cap A^2(x_3x_4) - 2_1$. Then we color both $x_1x_6, x_3x_4$ by $2_2$. Since $|A^2(x_5x_6)| \ge 3$ and $|A^2(x_2x_3)| \ge 3$, we can color each of them by a color in $A^2(x_5x_6) - \{2_1, 2_2\}$ and $A^2(x_2x_3) - \{2_1, 2_2\}$ respectively.
			
			\textbf{Case 1.2:} $A^2(x_1x_6) \cap A^2(x_3x_4) - 2_1 = \emptyset$. By an argument similar to Case 1.1, we may assume that $A^2(x_2x_3) \cap A^2(x_5x_6) - 2_1 = \emptyset$ as well. Let $\{\alpha_1,\alpha_2\} \subseteq A^2(x_1x_6)$ and $\{\alpha_3, \alpha_4\} \subseteq A^2(x_3x_4)$ such that $\{\alpha_1, \alpha_2\} \cap \{\alpha_3, \alpha_4\} = \emptyset$; let $\{\beta_1,\beta_2\} \subseteq A^2(x_2x_3)$ and $\{\beta_3, \beta_4\} \subseteq A^2(x_5x_6)$ such that $\{\beta_1, \beta_2\} \cap \{\beta_3, \beta_4\} = \emptyset$. Among the four of $\{\alpha_1, \alpha_3\}, \{\alpha_1, \alpha_4\}, \{\alpha_2, \alpha_3\}, \{\alpha_2, \alpha_4\}$ there is at least one pair $\{\alpha_i, \alpha_j\}$ such that $ \{\beta_1, \beta_2\} - \{\alpha_i, \alpha_j\} \neq \emptyset$ and $ \{\beta_3, \beta_4\} - \{\alpha_i, \alpha_j\} \neq \emptyset$. We color $x_1x_6$ with $\alpha_i$ and $x_3x_4$ with $\alpha_j$. We color $x_2x_3$ with a color in $\{\beta_1, \beta_2\} - \{\alpha_i, \alpha_j\}$ and  $x_5x_6$ with a color in $\{\beta_3, \beta_4\} - \{\alpha_i, \alpha_j\}$.
			
			\textbf{Case 2:} $A^2(x_1x_2) \cap A^2(x_4x_5) = \emptyset$. By symmetry, we can assume $A^2(x_1x_2) \cap A^2(x_4x_5) = A^2(x_2x_3) \cap A^2(x_5x_6) = A^2(x_3x_4) \cap A^2(x_1x_6) = \emptyset$. Recall that each $|A^2(x_ix_{i+1 (mod\ 6)})| \ge 3$, where $1 \le i \le 6$. So the union of the $A^2$ set of any pair of edges at distance three from $\{x_1x_2, x_2x_3, tx_3x_4, x_4x_5, x_5x_6, x_6x_1\}$ has size at least $6$. Since there is at least one pair of edges at distance three in the cycle $C$ for any collection of four edges of $\{x_1x_2, x_2x_3, x_3x_4, x_4x_5, x_5x_6, x_6x_1\}$, we can extend the coloring $f$ to $G$ by SDR.	
		\end{proof}

		
		
		
		All previous lemmas imply the following lemma.  For the sake of simplicity, in subsequent sections we will refer back to this lemma as opposed to the specific lemmas beforehand.
		
		\begin{lemma}\label{lem:girth}
			$G$ is a 3-regular, simple graph on at least 10 vertices, with girth at least seven, and no edge cuts of size at most two.
		\end{lemma}
		

		\section{Final Proof Using Edge Cuts}\label{sec:cuts}
		
		In this section, we will prove Theorem \ref{LSS1}. We begin by proving two useful results concerning edge cuts in $G$. For disjoint $X, Y \subseteq V(G)$, let $[X,Y]$ denote the set of edges in $G$ that have exactly one endpoint in $X$ and the other in $Y$.  In particular, we will be using this notation when $X \cup Y = V(G)$ so that $[X,Y]$ denotes an edge cut in $G$.

		\begin{lemma}\label{lem:3cut}
			Let $[X,Y] = \{x_1y_1, x_2y_2, x_3y_3\}$ be a nontrivial edge cut in $G$ such that $x_1, x_2, x_3 \in X$, and $y_1, y_2, y_3 \in Y$.  Then $[X,Y]$ is an induced matching in $G$, and without loss of generality, every good  coloring of $G - [X,Y]$ results in $U(x_1) = U(x_2) = U(x_3)$ and $|U(y_1) \cup U(y_2) \cup U(y_3)| = 6$.
		\end{lemma}
		
		\begin{proof}
			First note that if $[X,Y]$ is an edge cut of $G$, then it must be a matching.  If not, then say $x_1 = x_2$ is adjacent to a third vertex $x_4$ in $X$, and we can replace $[X,Y]$ with $\{x_1x_4, x_3y_3\}$ to obtain an edge cut of size two contradicting Lemma \ref{lem:cutedge}.  
			
			By the minimality of $G$, $G - [X, Y]$ has a good coloring such that each $x_i$ and $y_i$ is only incident to edges that are colored with 2-colors as they have degree two in $G - [X,Y]$.  In particular, since $G - [X,Y]$ consists of two components, namely $G[X]$ and $G[Y]$, we can consider the colorings on each of them individually.  Our goal is to show that in every case other than the one described in the statement, there is a permutation of the 2-colors on $G[X]$ and $G[Y]$ such that $U(x_i) \cap U(y_i) = \emptyset$ for each $i$.  We can then impose these colors onto $G$ and color the edges of $[X,Y]$ with 1 to obtain a good coloring of $G$.
			
			Suppose without loss of generality that $2_1 \in U(x_1) \cap U(x_2) \cap U(x_3)$.  We may also assume that $U(x_1) = \{2_1, 2_2\}, U(x_2) \subseteq \{2_1, 2_2, 2_3\}$, $U(x_3) \subseteq \{2_1, \dots, 2_4\}$.  If $|U(y_1) \cup U(y_2) \cup U(y_3)| \le 5$, then we can permute the 2-colors on $G[Y]$ so that $U(y_3) = \{2_6, 2_7\}, U(y_2) \subseteq \{2_4, \dots, 2_7\}$, and $U(y_1) \subseteq \{2_3, \dots, 2_7\}$.  If $|U(y_1) \cup U(y_2) \cup U(y_3)| = 6$ and we do not have $U(x_1) = U(x_2) = U(x_3)$, then we may assume $2_2 \in U(x_1) \cup U(x_2)$ but $2_2 \notin U(x_3)$ so that $U(x_3) \subseteq \{2_1, 2_3, 2_4\}$.  We then permute the 2-colors of $G[Y]$ so that $U(y_3) = \{2_2, 2_7\}$, $U(y_2) = \{2_5, 2_6\}$, and $U(y_1) = \{2_3, 2_4\}$.  Therefore if $2_1 \in U(x_1) \cap U(x_2) \cap U(x_3)$, then we must have $U(x_1) = U(x_2) = U(x_3)$ and $|U(y_1) \cup U(y_2) \cup U(y_3)| = 6$.
			
			In all the rest, we may suppose by symmetry that $U(x_1) \cap U(x_2) \cap U(x_3) = \emptyset$ and $U(y_1) \cap U(y_2) \cap U(y_3) = \emptyset$.  In addition this implies $|U(x_1) \cup U(x_2) \cup U(x_3)| \ge 4$ and $|U(y_1) \cup U(y_2) \cup U(y_3)|  \ge 4$.  Suppose $|U(x_1) \cup U(x_2) \cup U(x_3)| = 4$.  We may assume without loss of generality that say $2_1 \in U(x_1) \cap U(x_2)$ so that $2_1 \notin U(x_3)$, and further $U(x_1) = \{2_1, 2_2\}$, $U(x_2) \subseteq \{2_1, 2_2, 2_3\}$, and $U(x_3) \subseteq \{2_2, 2_3, 2_4\}$.  If $|U(y_1) \cup U(y_2) \cup U(y_3)| \le 5$, then we can permute the 2-colors in $G[Y]$ so that $U(y_3) = \{2_6, 2_7\}$, $U(y_2) \subseteq \{2_4, \dots, 2_7\}$, and $U(y_1) \subseteq \{2_3, \dots, 2_7\}$.  If $|U(y_1) \cup U(y_2) \cup U(y_3)| =6$, then we can permute the 2-colors in $G[Y]$ so that $U(y_3) = \{2_1, 2_7\}$, $U(y_2) = \{2_5, 2_6\}$, and $U(y_1) = \{2_3, 2_4\}$.
			
			So we may assume $|U(x_1) \cup U(x_2) \cup U(x_3)|  \ge 5$ and $|U(y_1) \cup U(y_2) \cup U(y_3)|  \ge 5$.  Suppose $|U(x_1) \cup U(x_2) \cup U(x_3)|  = 5$.  So without loss of generality, we may assume $U(x_1) = \{2_1, 2_2\}, U(x_2) = \{2_1, 2_3\}$, and $U(x_3) = \{2_4, 2_5\}$.  If $|U(y_1) \cup U(y_2) \cup U(y_3)| = 6$, then we can permute the 2-colors on $G[Y]$ so that $U(y_3) = \{2_6, 2_7\}$, $U(y_2) = \{2_2, 2_5\}$, and $U(y_1) = \{2_3, 2_4\}$.  If $|U(y_1) \cup U(y_2) \cup U(y_3)|  = 5$, then we can permute the 2-colors on $G[Y]$ so that $U(y_1) \cup U(y_2) \cup U(y_3) = \{2_3, \dots, 2_7\}$ where $2_7$ is the 2-color that appears in exactly two $U(y_i)$.  If $2_7 \in U(y_1) \cap U(y_2)$, then we use $U(y_1) = \{2_5, 2_7\}$, $U(y_2) = \{2_4, 2_7\}$, and $U(y_3) = \{2_3, 2_6\}$. If $2_7 \in U(y_1) \cap U(y_3)$, then we use $U(y_1) = \{2_3, 2_7\}$, $U(y_2) = \{2_4, 2_5\}$, and $U(y_3) = \{2_6, 2_7\}$.  The case of $2_7 \in U(y_2) \cap U(y_3)$ is symmetric to the previous.
			
			Therefore we may assume $|U(x_1) \cup U(x_2) \cup U(x_3)| = |U(y_1) \cup U(y_2) \cup U(y_3)|  = 6$.  Here we permute the 2-colors on $G[X]$ and $G[Y]$ so that $U(x_1) = \{2_1, 2_2\}$, $U(x_2) = \{2_3, 2_4\}$, $U(x_3) = \{2_5, 2_6\}$, $U(y_3) = \{2_4, 2_7\}$, $U(y_2) = \{2_2, 2_5\}$, and $U(y_1) = \{2_3, 2_6\}$.  This completes all cases and proves the lemma.
		\end{proof}

		\begin{lemma}\label{lem:4cut}
			Let $[X,Y] = \{x_1y_1, x_2y_2, x_3y_3, x_4y_4\}$ be an edge cut of $G$  such that $x_1, x_2, x_3, x_4 \in X$ and $y_1, y_2, y_3, y_4 \in Y$, and furthermore, assume that $[X,Y]$ is a matching.  If $G[X]$ has a good coloring such that $|U(x_1) \cup U(x_2) \cup U(x_3) \cup U(x_4)| = 5$ and $U(x_1) \cap U(x_2) \cap U(x_3) \cap U(x_4) = \emptyset$, then we can obtain a good coloring of $G$.
		\end{lemma}
		
		\begin{proof}
			The proof of this lemma utilizes SAGE.  A link to the proof can be found in the appendix.
		\end{proof}

		\subsection*{Setup}
		
		The goal of the following is to define a partition of $V(G)$ into three sets, say $L, M$, and $R$, such that $[L, M\cup R]$ and $[L \cup M, R]$ are edge cuts in $G$, where at least one of these edge cuts has size at most four.  We will then consider good colorings of $G[L]$ and $G[R]$, extend the good coloring of $G[L]$ (or $G[R]$) to a good coloring of $G[L \cup M]$ (or $G[L \cup R]$) in such a way that we can apply either Lemma \ref{lem:3cut} or Lemma \ref{lem:4cut} to finish.
		
		
		Fix a vertex $x \in V(G)$, and let $N_G(x) = \{u,v,w\}$.  For $a \in \{u,v,w\}$, let $N_G(a) = \{x, a_1, a_2\}$; for $i \in \{1,2\}$, let $N_G(a_i) = \{a, a_{i1}, a_{i2}\}$; for $i, j \in \{1,2\}$, let $N_G(a_{ij}) = \{a_i, a_{ij1}, a_{ij2}\}$.
		Note that by Lemma \ref{lem:girth}, each of these vertices are distinct from one another, and the graph induced by these vertices in $G$ is a tree, except for possibly edges of the form $a_{ij}b_{k\ell}$, where $a, b \in \{u, v, w\}$, $i, j, k, \ell \in \{1, 2\}$ and $a \neq b$.

		By Petersen's Theorem, we know that $G$ has a perfect matching, say $P$.  Color every edge in $P$ with 1.  As a result, every vertex is incident to exactly one edge colored with 1.  Without loss of generality, suppose $xw, uu_1, vv_1,  u_2u_{21}, v_2v_{21}, w_1w_{11}, w_2w_{21}$ are colored with 1.  We now adjust this coloring by uncoloring $xw$ and $w_2w_{21}$, and then coloring $ww_2$ with 1. Note that $x$ and $w_{21}$ are the only vertices in $G$ that are not incident to an edge colored with 1.  We also color $uu_2, vv_2$, and $ww_1$ each with $2_1$. By Lemma \ref{lem:girth}, this is a good partial coloring of $G$, and Lemmas \ref{lem:uncolored} - \ref{lem:AF} will always assume this good partial coloring of $G$.  See Figure \ref{fig:tree-partialcoloring}.
		
		
		\begin{figure}[h]
			\centering
			\includegraphics{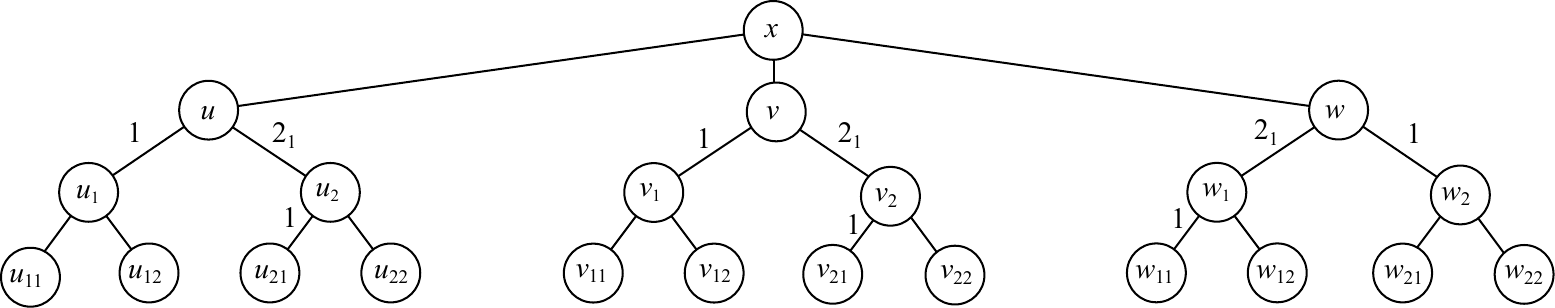}
			\caption{A good partial coloring of $G$.}
			\label{fig:tree-partialcoloring}
		\end{figure}
		
		\begin{lemma}\label{lem:uncolored}
			All but three uncolored edges in $G$ see exactly four edges colored with 1.  The remaining three uncolored edges each see exactly three edges colored with 1, and in particular, these three remaining uncolored edges are $w_2w_{22}$, the uncolored edge incident to $w_{211}$ other than $w_{21}w_{211}$, and the uncolored edge incident to $w_{212}$ other than $w_{21}w_{212}$.  
		\end{lemma}
		
		\begin{proof}
			Let $ab \in E(G)$ such that $ab$ is uncolored, and suppose that $ab$ sees at most three edges colored with 1.  Consider $A = N_G(a) \setminus \{b\}$ and $B = N_G(b) \setminus\{a\}$.  By Lemma \ref{lem:girth}, these sets are disjoint.  Therefore,  if every vertex in $A$ and every vertex in $B$ is incident to an edge colored with 1, then $ab$ must see exactly four edges colored with 1.  Since the only vertices in $G$ that are not incident to an edge colored with 1 are $x$ and $w_{21}$, we must have $x$ or $w_{21}$ in $A \cup B$.
			
			Suppose $x \in A \cup B$.  If $x \in A$, then note that Lemma \ref{lem:girth} implies $w_{21} \notin \{a,b\} \cup A$.   If $w_{21} \in B$, then again by Lemma \ref{lem:girth}, $ab = ww_2$; however $ww_2$ is already colored with 1 and $ab$ is assumed to be uncolored.  A symmetric argument holds if $x \in B$.
			
			Therefore, $x \notin A \cup B$, which implies that we may only have $w_{21} \in A \cup B$ so that $ab$ sees exactly three edges colored with 1.  Without loss of generality, if $w_{21} \in A$, then $a \in \{w_2, w_{211}, w_{212}\}$.  Note that if $a = w_2$, then $ab = w_2w_{22}$.  Suppose $a = w_{211}$.  Recall that every vertex other than $x$ or $w_{21}$ is incident to an edge colored with 1.  By Lemma \ref{lem:girth}, $w_{211} \neq x$ so that $w_{211}$ is incident to an edge colored with 1. This implies $ab$ is the uncolored edge incident to $w_{211}$ other than $w_{21}w_{211}$, so $ab$ sees exactly three edges colored with 1 as claimed. A similar argument holds if $a = w_{212}$.  This proves the lemma.
		\end{proof}

		Define a sequence of edges $S_0 = w_{21}w_{211}, w_{21}w_{212}, w_2w_{22}, w_2w_{21}, xu, xv, xw$.  We will extend $S_0$ to a sequence $S$ such that all the following hold:
		\begin{enumerate}
			\item $S_0$ is at the end of $S$ (i.e.,  the last edges of $S$ will consists of $S_0$ in this exact order);
			\item for every edge $e \in S - S_0$, either $e$ sees four edges colored with 1 and sees at least two edges in $S$ that come later in the sequence of $S$ than $e$, or $e$ sees three edges colored with 1 and sees at least three edges in $S$ that come later in the sequence of $S$ than $e$; 
			\item  among all sequences satisfying (1) and (2), $S$ is longest.
		\end{enumerate}

		Note that $S - S_0$ is nonempty as the uncolored  edge incident to $w_{22}$ other than $w_2w_{22}$ is in $S - S_0$, as is the uncolored edge incident to  $w_{211}$ other than $w_{21}w_{211}$, and the uncolored edge incident to $w_{212}$ other than $w_{21}w_{212}$.  Therefore, by Lemma \ref{lem:uncolored} and defining $w_2w_{22}$ in $S_0$,  every uncolored edge that is not in $S$ sees exactly four edges colored with 1. 
		
		Furthermore, if $e$ is an uncolored edge of $G$ that is not in $S$, then $e$ cannot see two edges in $S$, as otherwise we could create a new sequence from $S$, by putting $e$ at the start of $S$.  This new sequence, say $S'$, would satisfy (1) and (2), but would be longer than $S$, contradicting (3) (i.e., the maximality of $S$).

		\begin{lemma}\label{lem:S=G}
			$S$ does not contain all of the uncolored edges of $G$.
		\end{lemma}
		
		\begin{proof}
			Suppose on the contrary that $S$ contains all of the uncolored edges of $G$.  We will color all of the uncolored edges of $G$ in the order they appear in $S$ with possibly changing the coloring on $ww_1$.  This will result in a good coloring of $G$, which is a contradiction to our assumption of $G$ as a counterexample.
			
			Let $e \in S$ and assume we have colored every edge in $S$ that appears before $e$ in such a way that we have a good partial coloring of $G$.    If $e$ sees four edges colored with 1 and at least two edges in $S$ that appear later in the sequence of $S$ than $e$ (i.e., they are not yet colored), then $e$ will see at most  $12 - 4 - 2 = 6$ different 2-colors, and so we can color $e$ with a 2-color and extend our good partial coloring of $G$.  If $e$ sees three edges colored with 1 and at least three edges in $S$ that appear later in the sequence of $S$ than $e$, then $e$ will see at most $12 - 3 - 3 = 6$ different 2-colors, and again we can color $e$ with a 2-color and extend our good partial coloring of $G$.
			
			Therefore every edge in $S - S_0$ can be greedily colored in this way. Consider $w_{21}w_{211}$ and $w_{21}w_{212}$.  Neither edge will be $uu_2, vv_2$ or $ww_2$ by Lemma \ref{lem:girth}.  We can greedily color $w_{21}w_{211}$ as it sees three edges colored with 1 and three edges in $S$ that appear later in the sequence of $S$.  Note that $w_{21}w_{212}$ now sees four edges colored with 1 and at most six other colored edges.  Therefore, we can further extend our good partial coloring to $w_{21}w_{212}$.  Call this good partial coloring of $G$, $\phi$.
			
			Note that the only edges in $G$ that are currently uncolored are $w_2w_{22}, w_2w_{21}, xu, xv, xw$, and furthermore, $|A^2_\phi(w_2w_{22})| \ge 0, |A^2_\phi(w_2w_{21})| \ge 1, |A^2_\phi(xu)| \ge 3, |A^2_\phi(xv)| \ge 3$, and $|A^2_\phi(xw)| \ge 5$.  
			
			Suppose that $|A^2_\phi(w_2w_{22})| \ge 1$ so that we can further extend our good partial coloring by coloring $w_2w_{22}$ with say $2_7$.  If we can also color $w_2w_{21}$, then we can easily color $xu, xv, xw$ in this order.  If we cannot color $w_2w_{21}$, then it must be that $A^2_\phi(w_2w_{21}) = \{2_7\}$.  However in $\phi$, $w_2w_{21}$ sees exactly six edges colored with 2-colors, which means that each of these six edges must be colored with $2_1, \dots, 2_6$.  In particular, $ww_1$ must be the only edge seen by $w_2w_{21}$ that is colored with $2_1$.  Therefore, we can uncolor $ww_1$, and instead color $w_2w_{22}$ with $2_7$ and color $w_2w_{21}$ with $2_1$ to obtain a new good partial coloring, say $\psi$.  Note that $|A^2_\psi(ww_1)| \ge 1$, $|A^2_\psi(xu)| \ge 3$, $|A^2_\psi(xv)| \ge 3$, and $|A^2_\psi(xw)| \ge 4$, and we can greedily color $ww_1, xu, xv, xw$ in this order.
			
			Therefore we must assume $|A^2_\phi(w_2w_{22})| = 0$.  Now as with $w_2w_{21}$ above, $w_2w_{22}$ sees exactly seven edges (in $\phi$) colored with 2-colors, and they must be colored with $2_1, \dots, 2_7$.  In particular, $ww_1$ is the only edge seen by $w_2w_{22}$ that is colored with $2_1$.  Therefore, we can uncolor $ww_1$, and instead color $w_2w_{22}$ with $2_1$, to obtain a new good partial coloring, call it $\rho$.  Note that $|A^2_\rho(w_2w_{21})| \ge 1$, $|A^2_\rho(ww_1)| \ge 2$, $|A^2_\rho(xu)| \ge 3$, $|A^2_\rho(xv)| \ge 3$, and $|A^2_\rho(xw)| \ge 5$, and we can greedily color $w_2w_{21}, ww_1, xu, xv, xw$ in this order.
		\end{proof}
		
		Let $H$ be the set of uncolored edges in $G$ that are not in $S$.  As noted above, Lemma \ref{lem:uncolored} implies every edge in $H$ sees four edges colored with 1.  Let $L \subseteq V(G)$ such that $L$ is incident to an edge in $H$.  By Lemma \ref{lem:S=G}, $H$ is nonempty, and so $L$ is nonempty.
		
		Note that $x, u, v, w$ are not in $L$, as each are only incident to edges that are colored or are in $S$.

		\begin{lemma}
			No edge in $S$ has both endpoints in $L$.
		\end{lemma}
		
		\begin{proof}
			Let $ab \in S$ such that $a, b \in L$, and let $N_G(a) = \{b, a_1, a_2\}$ and $N_G(b) = \{a, b_1, b_2\}$.  Since $a, b \in L$, $a$ and $b$ must both be incident to an uncolored edge that is not in $S$.  Without loss of generality, suppose $aa_1, bb_1$ are such edges (i.e., they are in $H$).    Recall that every uncolored edge not in $S$ sees four edges colored with 1.
			
			By the definition of $S$, $ab$ is an uncolored edge of $G$, where either $ab \in S_0$, or $ab$  sees four edges colored 1 and at least two uncolored edges after it in $S$, or $ab$ sees three edges colored 1 and at least three uncolored edges after it in $S$.  
			
			Note that $ab \notin S_0$ as every edge in $S_0$ is incident to either $x, w_{21}$ or $w_2$.  Each of these vertices is surrounded by edges that are in $S$ or colored, and so are not incident to any uncolored edge not in $S$ (i.e., $H$).  
			
			Therefore $ab$ must see at least two uncolored edges after it in $S$.  If one such edge is incident to $a_1$ or $a_2$, then $aa_1$ will see at least two uncolored edges in $S$, this edge and $ab$.  So we can extend $S$ by putting $aa_1$ at the start, contradicting the maximality of $S$.  However, a symmetric argument applies to $bb_1$ if an edge of $S$ is incident to either $b_1$ or $b_2$.
		\end{proof}

		Let $F = [L, V(G)\setminus L]$.  A vital component in our subsequent proofs will involve understanding $F$.  In particular, we will consider case by case the possibilities for $F$.  We know by Lemma \ref{lem:3cut} that $|F| \ge 3$ and if $|F| = 3$, then $F$ must be an induced matching.  In addition, the following two lemmas will further restrict the cases we need to consider.

		\begin{lemma}\label{lem:F}
			$F \subseteq \{uu_1, uu_2, vv_1, vv_2, ww_1, u_2u_{21}, u_2u_{22}, v_2v_{21}, v_2v_{22}, w_1w_{11}, w_1w_{12}\}$.
		\end{lemma}
		
		\begin{proof}
			Let $ab \in F$, where $a \in L$ and $b \in V(G)\setminus L$.  Let $N_G(a) = \{b, a_1, a_2\}$ and $N_G(b) = \{a, b_1, b_2\}$.  Since $b \notin L$, we know $ab \notin H$, and hence either $ab \in S$ or $ab$ is already colored.  
			
			Suppose $ab$ is already colored.  If $ab$ is colored $2_1$, then $ab \in \{uu_2, vv_2, ww_1\}$.  
			
			If $ab$ is colored 1, then since $a \in L$, $a$ must be incident to an uncolored edge that is not in $S$ (i.e., is in $H$).  Suppose $aa_1$ is such an edge.  Since $aa_1$ is not in $S$ and is uncolored, it must see exactly four edges colored with 1.  Therefore we cannot have both $bb_1, bb_2 \in S$, as otherwise $aa_1$ would see two edges in $S$ and could then be added to the start of $S$, contradicting the maximality of $S$.  So without loss of generality, $bb_1$ is either uncolored and not in $S$, or is already colored.  If $bb_1$ is already colored, then since we are assuming $ab$ is already colored and it is with 1, then $bb_1$ must be colored with $2_1$ so that $b \in \{u, v, w, u_2, v_2, w_1\}$.  However, $a \neq w_2$ as $w_2$ is incident to only colored edges and edges in $S$ so that $w_2 \notin L$.  Thus in this case, $ab \in \{uu_1, u_2u_{21}, vv_1, v_2v_{21}, w_1w_{11}\}$.
			
			Lastly suppose $ab$ is in $S$.  As above, since $a \in L$, we may assume $aa_1$ is an uncolored edge not in $S$.  Thus, $ab$ sees four edges colored with 1.  If either $bb_1$ or $bb_2$ is in $S$, then again with contradict the maximality of $S$ as $aa_1$ will see this edge as well as $ab$.  Therefore both $bb_1$ and $bb_2$ must be already colored with say 1 and $2_1$, respectively.  This implies $b \in \{u, u_2, v, v_2, w, w_1\}$.  Note that $b \notin \{u, v, w\}$ as it would imply $a = x$, however $x \notin L$ as $x$ is only incident to edges that are in $S$.  So $b \in \{u_2, v_2, w_1\}$ and as $ab$ is uncolored (it is in $S$), we must have $ab \in \{u_2u_{22}, v_2v_{22}, w_1w_{12}\}$.
		\end{proof}

		\begin{lemma}\label{lem:AF}
			Let $A = \{u_1, u_2, v_1, v_2, w_1\}$.   Every vertex in $A \cup L$ is incident to at most one edge from $F$.  
		\end{lemma}

		\begin{proof}
			In the following, let $a \in A \cup L$ with $N_G(a) = \{a_1, a_2, a_3\}$, and suppose on the contrary that $a$ is incident to at least two edges from $F$, say $aa_1, aa_2 \in F$.
			
			Suppose $a \in L$.  Since $F = [L, V(G)\setminus L]$ and $a \in L$, we must have $a_1, a_2 \in V(G)\setminus L$.  Therefore every edge incident to either $a_1$ or $a_2$ is already colored or is uncolored and in $S$.  Again since $a \in L$,  $a$ is incident to an uncolored edge that is not in $S$ (i.e., is in $H$), whose other endpoint must then be in $L$.  So $a_3 \in L$ and $aa_3 \in H$.  
			
			Now $a_1, a_2$ are incident to only edges that are already colored or are in $S$.  For each $a_i$, at most two of these edges can be colored, which means that each is incident to an edge in $S$.  This implies $aa_3$ sees two edges in $S$, and since $aa_3 \in H$, we know that $aa_3$ also sees four edges colored with 1.  Yet this contradicts the maximality of $S$ as we could place $aa_3$ at the start of $S$.
			
			Now suppose $a \in A$.   By Lemma \ref{lem:F},  we must have $a \in \{u_2, v_2, w_1\}$. Note that if $aa_3 \in F$, then because $u, v, w \notin L$, we must have $a \in L$ as $F = [L, V(G)\setminus L]$.  However, this contradicts the above.  So $aa_3 \notin F$.  In particular, if $a$ is $u_2, v_2, w_1$, respectively, then $a_3$ is $u, v, w$, respectively.  Furthermore, we can assume without loss of generality that $aa_1$ is colored with 1 and $aa_2$ is uncolored.
			
			Again if $a \in L$, then we are done by the above.  So we must have $a \in V(G)\setminus L$, and $a_1, a_2 \in L$.  In addition, since $a \notin L$, we know $aa_2$ is in $S$ as it is uncolored and has an endpoint not in $L$, namely $a$.  
			
			Since $a_1, a_2 \in L$, they must be incident to an uncolored edge that is not in $S$.  For $i \in \{1,2\}$, let $N_G(a_i) = \{a, a_{i1}, a_{i2}\}$.  Without loss of generality, suppose $a_1a_{11}$ and $a_2a_{21}$ are uncolored edges not in $S$.  Each of these edges see four edges colored with 1, and they each see $aa_2$, which is an edge in $S$.  We will show that at least one of these edges will also see a second edge in $S$, getting a contradiction to the maximality of $S$, as we could add this edge to the start of $S$.
			
			Note that no edge incident to $u_2, v_2$ or $w_1$ is in $S_0$ by Lemma \ref{lem:girth}, so that $aa_2 \notin S_0$.  Therefore, $aa_2 \in S - S_0$ so that $aa_2$ must see at least two edges in $S$.  If $a$ is $u_2, v_2, w_1$, respectively, then one such edge is $xu, xv, xw$, respectively.  The other edge must be incident to either $a_1$, $a_{21}$, or $a_{22}$.  If it is incident to $a_1$, then $a_1a_{11}$ will see it, and if it is incident to $a_{21}$ or $a_{22}$, then $a_2a_{21}$ will see it.  In either case, we contradict the maximality of $S$.
		\end{proof}

		Now let $V_F$ be the set of vertices in $V(G)\setminus L$ that are incident to an edge from $F$.  We are now ready to partition $V(G)\setminus L$ into sets $M$ and $R$.  If $w$ or $w_1$  is in $V_F$, then let  $M = \{x, u, v, w\} \cup V_F$, and if neither $w$ nor $w_1$ is in $V_F$, then let $M = \{x, u, v\} \cup V_F$.  In either case, let $R = V(G)\setminus (L \cup M)$, and let $F' = [M, R]$. See Figure \ref{fig:tree} for an example.
		
		\begin{figure}[h]
			\centering
			\includegraphics{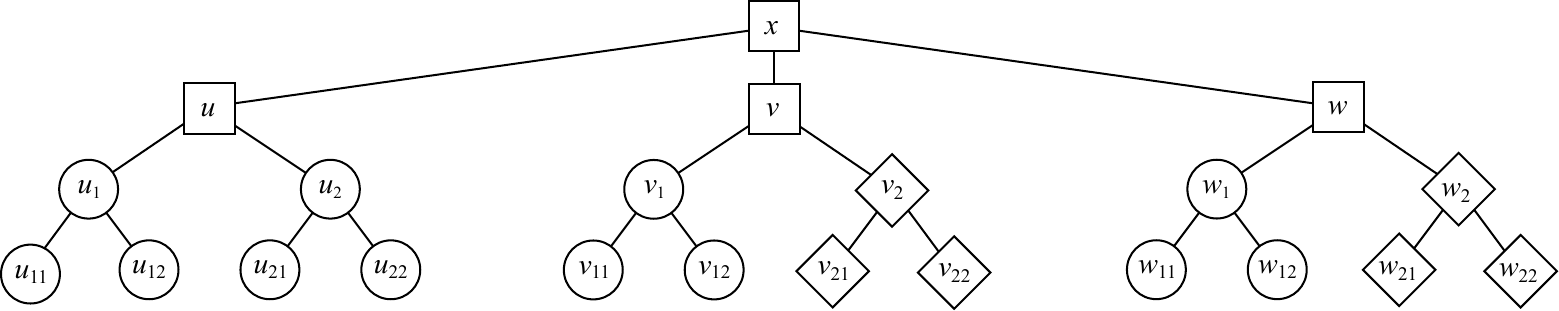}
			\caption{A possible partition of the vertices with $M=\{x,u,v,w\}$ (square vertices), $L$ (circle vertices), and $R$ (diamond vertices), where $F=\{uu_1, uu_2, vv_1, ww_1\}$.}
			\label{fig:tree}
		\end{figure}
		
		At this point, we now remove all colorings on $E(G)$ and resort only to structural arguments.  As a result, we can now appeal to symmetry even when involving $w$ and $w_1$ (e.g., $F = \{uu_2, v_2v_{21}, w_1w_{12}\}$ is the same as $F = \{u_2u_{22}, v_2v_{21}, ww_1\}$).  
		
		Note that every edge of $F$ appears in exactly one of the following three sets: $F_U = \{uu_1, uu_2, u_2u_{21}, u_2u_{22}\}, F_V = \{vv_1, vv_2, v_2v_{21}, v_2v_{22}\}$, and $F_W = \{ww_1, w_1w_{11}, w_1w_{12}\}$.  By Lemma \ref{lem:AF}, $|F \cap F_U| \le 2, |F \cap F_V| \le 2$, and $|F \cap F_W| \le 1$, so that $|F| \le 5$.  Further, $|F \cap F_U| \ge 1$ as otherwise $|F| \le 3$, and by Lemma \ref{lem:3cut} we must have equality and $F$ must be an induced matching; however this is impossible if $|F \cap F_V| =2$.  Similarly, we must have $|F \cap F_V| \ge 1$.  Note that if $|F \cap F_W| = 0$, then $w \in R$.

		
		\begin{lemma}\label{lem:uu1uu2}
			If $uu_1, uu_2 \in F$, then without loss of generality, we may assume $F = \{uu_1, uu_2, v_2v_{21}, w_1w_{11}\}$.
		\end{lemma}
		
		\begin{proof}
			Suppose on the contrary that $uu_1, uu_2 \in F$.  By Lemma \ref{lem:AF} no other edge from $F_U$ is in $F$.   If both $vv_1, vv_2 \in F$ as well, then we can replace $uu_1, uu_2$ with $xu$ and $vv_1, vv_2$ with $xv$  in $F$ to obtain a new edge cut.  Since $|F| \le 5$ by Lemma \ref{lem:AF}, this implies this new edge cut has size at most 3.  However, by Lemma~\ref{lem:3cut} it must have size exactly 3 and be an induced matching, which it is not as both $xu, xv$ are in this edge cut.
			
			Similarly, we cannot have $F \cap F_W = \emptyset$ as otherwise $|F| \le 4$ and we can replace $uu_1, uu_2$ with $xu$ in $F$ to obtain a new edge cut.  This edge cut will have size at most 3 and contain two edges from $F_V$, which will mean that this new edge cut cannot be an induced matching, contradicting Lemma \ref{lem:3cut}.  Furthermore, $F \cap F_V \neq \emptyset$, as otherwise $|F| = 3$ and we again contradict Lemma \ref{lem:3cut}.
			
			So suppose $ww_1 \in F$ so that $ww_2 \in F'$.   If $|F \cap F_V| = 2$, then as we have already shown we cannot have both $vv_1, vv_2 \in F$, without loss of generality we must have $vv_1, v_2v_{21} \in F$.  However ths implies $F' = \{v_2v_{22}, ww_2\}$ contradicting  Lemma~\ref{lem:girth}.  Similarly, if $F \cap F_V = \{vv_1\}$ (or symmetrically $\{vv_2\}$), then $F' = \{vv_2, ww_2\}$, again contradicting Lemma \ref{lem:girth}.  Similarly, if $F \cap F_V = \{v_2v_{21}\}$ (or symmetrically $\{v_2v_{22}\}$), then $F' = \{vv_1, v_2v_{22}, ww_2\}$, which again contradicts Lemma \ref{lem:3cut} as it is not an induced matching.

			So without loss of generality, $w_1w_{11} \in F$ so that $w_1w_{12}, ww_2 \in F'$, and $w, w_1 \in M$.  As above if either $|F \cap F_V| = 2$, or if $F \cap F_V$ is either $\{vv_1\}$ or $\{vv_2\}$, then we get $|F'| =3$, but $F'$ will not be an induced matching because of $w_1w_{12}, ww_2 \in F'$, contradicting Lemma~\ref{lem:3cut}.  So we may assume without loss of generality that $F = \{uu_1, uu_2, v_2v_{21}, w_1w_{11}\}$ so that $F' = \{vv_1, v_2v_{22}, ww_2, w_1w_{12}\}$.  
		\end{proof}
		
		
		\subsection*{Final Cases}
		
		We will now finish our proof of Theorem \ref{LSS1} using cases based on what $F$ can be.  First note that we cannot have $|F|\le 3$ with $uu_1, uu_2 \in F$, as this would give an edge cut of size at most 3 that is not an induced matching (or an edge cut of size 2). For the same reason, we cannot have $|F|\le 3$ with $uu_1, u_2u_{21} \in F$. The cases that remain will follow the same general idea.  We first use the minimality of $G$ to obtain good colorings of $G[L]$ and $G[R]$.  We then consider $G[L \cup  M]$ (or $G[M \cup R]$) and adjust the coloring slightly to get a new good partial coloring.  Lastly, we extend this to a good coloring of $G[L\cup M]$ (or $G[M \cup R]$) in such a way that we can apply Lemma \ref{lem:3cut} to obtain a contradiction or use Lemma \ref{lem:4cut} to finish a good coloring of $G$.
		
		To be more precise, suppose we intend to extend the coloring of $G[L]$ to $G[L \cup M]$ (if we extend from $G[R]$ to $G[M \cup R]$, then we switch $L$ with $R$, and $F$ with $F'$ in the following).  We focus only on $G[L\cup M]$ and the good partial coloring imposed on it from our good coloring of $G[L]$.  We then color an edge from $F$ with 1 if and only if it is not incident to an edge from $F'$ in $G$ (when extending the coloring from $G[L]$ to $G[L \cup M]$ we will ultimately color every edge in $F'$ with 1).  
		
		We then perform an `adjustment' in our coloring of $G[L]$ as follows.  For each $a \in L$ such that $a$ is incident to an edge from $F$ (call this edge $ab$) that is not colored with 1, we consider the two edges incident to $a$ that are not in $F$ (both $F$ and $F'$ will always be matchings in the following).  Note that these two edges must be colored with 2-colors as $a$ had degree two in $G[L]$.  The adjustment is to recolor one of these two edges with  1 if possible and still keep a good partial coloring of $G[L \cup M]$.   If we can perform this recoloring,  $ab$ will now have at least two available 2-colors that we can color it with.  If we cannot perform this recoloring, then the two neighbors of $a$ in $L$ must both be incident to an edge colored with 1, which implies $ab$ will have  at least three available 2-colors that we can color it with.  So we may always assume that after adjustment, all such edges will have at least two 2-colors available.
		
		Furthermore, if $bc$ is an edge where $c \in M$, then if we can recolor an edge incident to $a$ with 1, then $bc$ will see one less edge in $G[L \cup M]$ that is colored with a 2-color.  In particular,
		\begin{equation}\label{eq:adjust}
			\text{if } bc  \text{ sees two edges colored with a 2-color incident to } a, \text{ then } ab \text{ will have at least three available 2-colors.}
		\end{equation}
		This fact about the adjustments will be very useful in the following cases. 
		
		\begin{figure}
			\centering
			\includegraphics[scale=0.7]{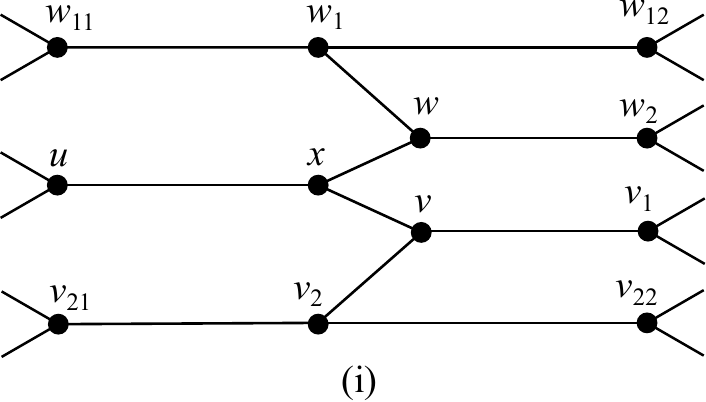} \hspace{0.5in} \includegraphics[scale=0.7]{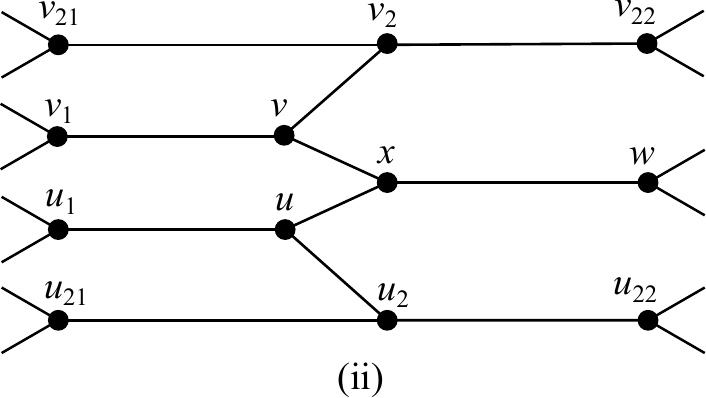}
			
			\vspace{.25in}
			\includegraphics[scale=0.7]{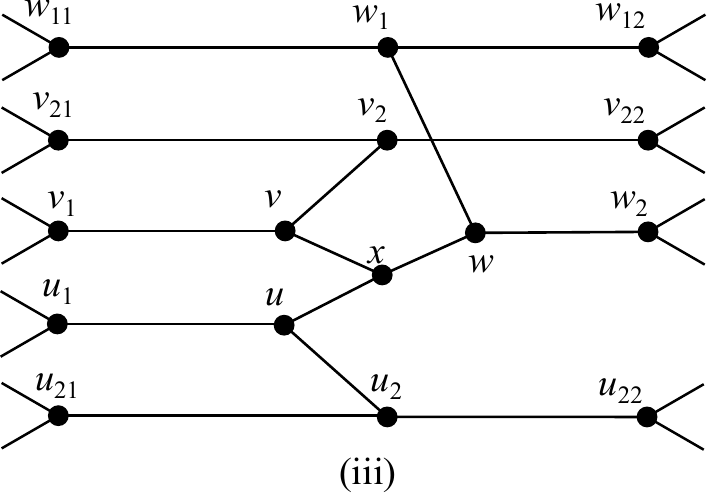} \hspace{0.5in} \includegraphics[scale=0.7]{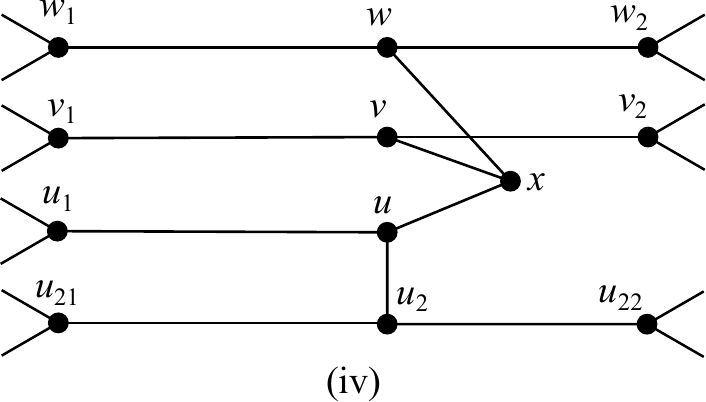}
			
			\vspace{.25in}
			\includegraphics[scale=0.7]{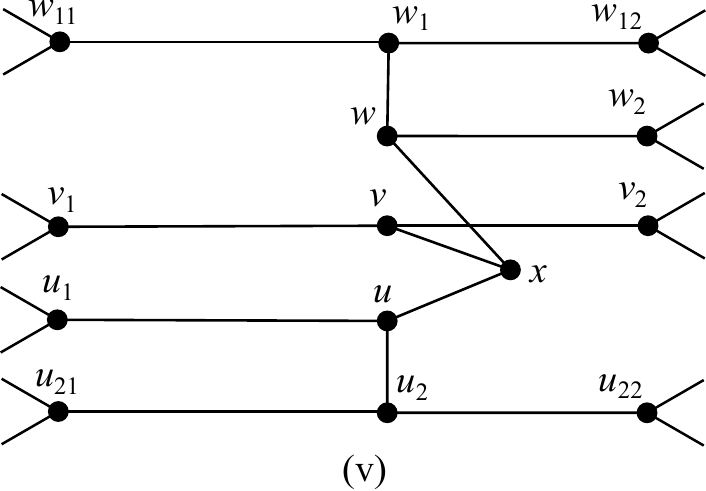} \hspace{0.5in} \includegraphics[scale=0.7]{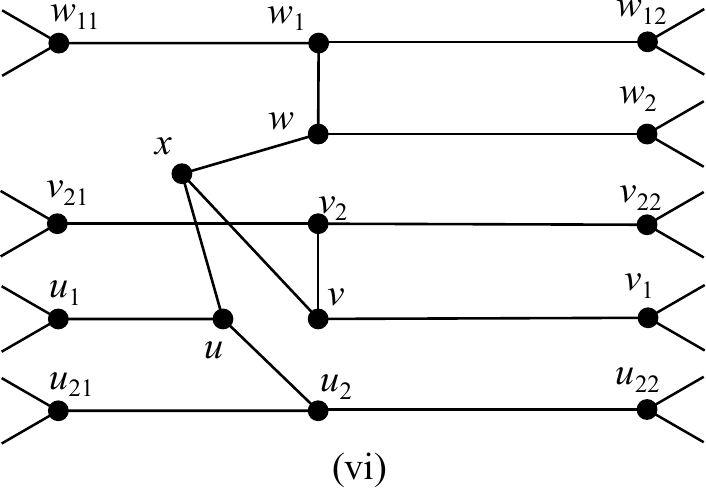}
			
			\vspace{.25in}
			\includegraphics[scale=0.7]{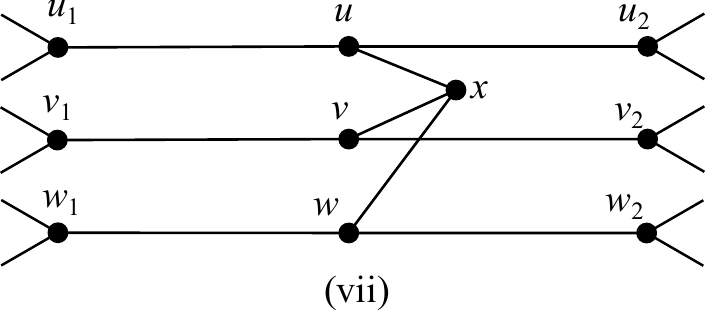} \hspace{0.5in} \includegraphics[scale=0.7]{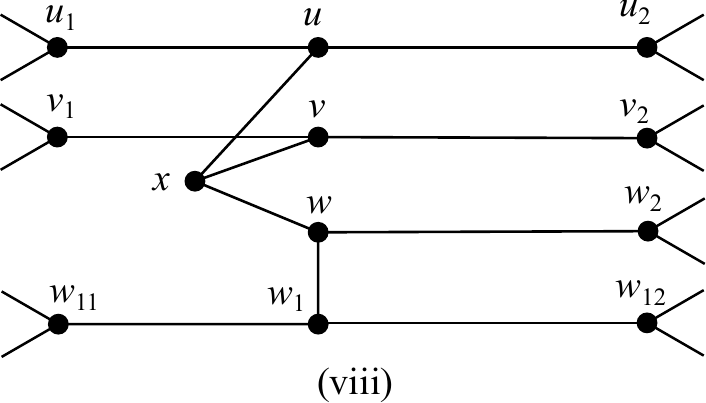}
			
			\vspace{.25in}
			\includegraphics[scale=0.7]{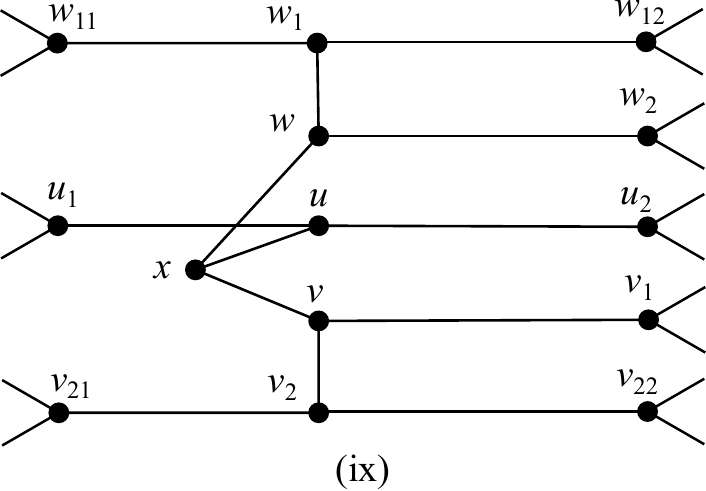} \hspace{0.5in} \includegraphics[scale=0.7]{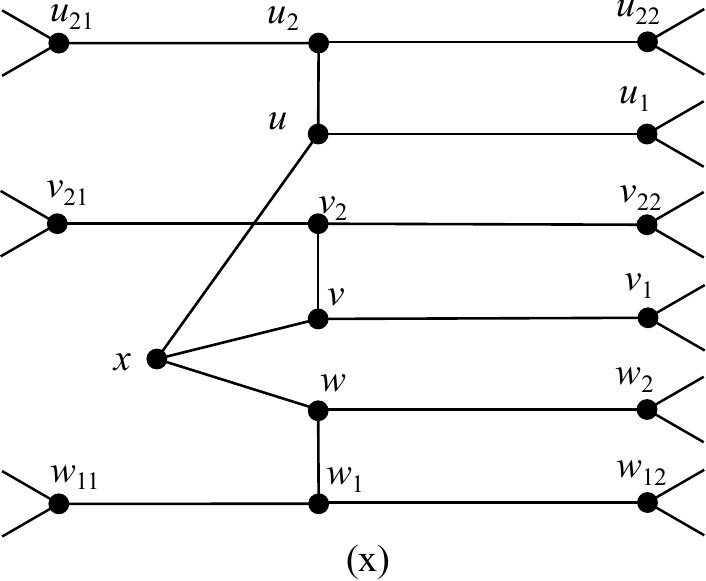}
			
			\caption{Final cases on the edge cut $F$ (up to symmetry).}
			\label{fig:finalcases}
		\end{figure}

		
		\begin{newcase}\label{case:uu1uu2}
			$F = \{uu_1, uu_2, v_2v_{21}, w_1w_{11}\}$.
		\end{newcase}
		
		\begin{proof}
			We now replace $L$ and $M$  with $L'$ and $M'$, respectively, where $L' = L \cup \{u\}$ and $M' = M \setminus \{u\}$.  As a result $[L', M' \cup R] = \{xu, v_2v_{21}, w_1w_{11}\}$, as seen in Figure \ref{fig:finalcases}(i).  We will now proceed in a manner that we will continue to use in subsequent cases.  
			
			By the minimality of $G$, both $G[L']$ and $G[R]$ have good colorings.  Impose the good coloring of $G[R]$ onto $G[M' \cup R]$ and color $ww_2, vv_1$ with 1.  We perform the adjustment on the two colored edges incident to $w_{12}$; that is, if we can recolor either of these two edges with a 1, then we do so.   We repeat this adjustment with  $v_{22}$.  Call this new good partial coloring of $G[M' \cup R]$, $\phi$.

			So $|A^2_\phi(w_1w_{12})| \ge 2, |A^2_\phi(v_2v_{22})| \ge 2, |A^2_\phi(ww_1)| \ge 3, |A^2_\phi(vv_2)| \ge 3, |A^2_\phi(xw)| \ge 5$, and $|A^2_\phi(xv)| \ge 5$.  If we can color both $vv_2$ and $ww_1$ the same, say with $2_1$, then we can color $w_1w_{12}, v_2v_{22}, xw, xv$ in this order.  This will yield a good coloring of $G[M' \cup R]$ that contradicts Lemma \ref{lem:3cut}, as the edges in $G[M' \cup R]$ that are incident to the edges in $[L', M' \cup R]$ will not be colored with all different colors ($ww_1$ and $vv_2$ are colored the same), nor will they have only two colors ($w_1w_{12}, ww_1, xw$ must each get different colors).
			
			So we cannot color both $vv_2$ and $ww_1$ the same.  This implies that without loss of generality, $|A^2_\phi(vv_2)|= 3$, which by (\ref{eq:adjust}) implies that $|A^2_\phi(v_2v_{22})| \ge 3$.  So we can color $v_2v_{22}$ and $xw$ the same, say with $2_1$.  We then color $w_1w_{12}, ww_1, vv_2, xv$ in this order to obtain a good coloring of $G[M' \cup R]$ that again contradicts Lemma \ref{lem:3cut}.
		\end{proof}
		
		Case \ref{case:uu1uu2},  Lemma \ref{lem:uu1uu2}, and symmetry, show that $F$ must be a matching in $G$ and furthermore, if $|F \cap F_U| =2$, then without loss of generality, $F \cap F_U = \{uu_1, u_2u_{21}\}$, and symmetrically, if $|F \cap F_V| = 2$, then without loss of generality, $F \cap F_V = \{vv_1, v_2v_{21}\}$.
		
		
		\begin{newcase}
			$F=\{uu_1, u_2u_{21}, vv_1, v_2v_{21}\}$.
		\end{newcase}
		
		\begin{proof}
			We have $F'=\{u_2u_{22}, v_2v_{22}, xw\}$ and $M=\{x,u,v,u_2,v_2\}$ as shown in Figure \ref{fig:finalcases}(ii). By the minimality of $G$, both $G[L]$ and $G[R]$ have good colorings. So we impose the good coloring of $G[L]$ onto $G[L\cup M]$, color $vv_1$ and $uu_1$ with 1, and perform the adjustment on the colored edges incident to $u_{21}$ and $v_{21}$. Note that after performing the adjustment, we must have $|A^2(u_2u_{21})|\ge 3$, or $|A^2(v_2v_{21})|\ge 3$, or both $|A^2(uu_2)|\ge 4$ and $|A^2(vv_2)|\ge 4$. Our goal will be to color the edge set $H=\{v_2v_{21}, vv_2, xv, xu, uu_2,u_2u_{21}\}$ with at most 5 colors, which contradicts Lemma~\ref{lem:3cut}.
			
			Suppose that $|A^2(u_2u_{21})|\ge 3$. Since $|A^2(xv)|\ge 5$, then we can color $u_2u_{21}$ and $xv$ with the same color, say $2_1$. Then we can finish greedily by coloring $v_2v_{21}$, $vv_2$, $uu_2$, and $xu$ in this order. 
			
			Now suppose that $|A^2(v_2v_{21})|\ge 3$. Since $|A^2(xu)|\ge 5$, then we can color $v_2v_{21}$ and $xu$ with the same color, say $2_1$. Then we can finish greedily by coloring $u_2u_{21}$, $uu_2$, $vv_2$, and $xv$ in this order. 
			
			Else, we must have $|A^2(uu_2)|\ge 4$ and $|A^2(vv_2)|\ge 4$. Then we can color $uu_2$ and $vv_2$ with the same color, say $2_1$. Then we can finish greedily by coloring $u_2u_{21}$, $v_2v_{21}$, $xu$, and $xv$ in this order.
		\end{proof}
		
		
		\begin{newcase}
			$F=\{uu_1, u_2u_{21}, vv_1, v_2v_{21}, w_1w_{11}\}$.
		\end{newcase}
		
		\begin{proof}
			We have $F'=\{u_2u_{22}, v_2v_{22}, w_1w_{12}, ww_2\}$ and $M=\{x, u, v, w, u_2, v_2, w_1\}$ as in Figure \ref{fig:finalcases}(iii). By the minimality of $G$, both $G[L]$ and $G[R]$ have good colorings. Now impose the coloring of $G[L]$ onto $G[L\cup M]$ and color $vv_1$ and $uu_1$ with 1. We perform the adjustment on the colored edges incident to $u_{21}$, $v_{21}$, and $w_{11}$ so that $|A^2(u_2u_{21})|, |A^2(v_2v_{21})|, |A^2(w_1w_{11})|\ge 2$. Call this good partial coloring $\phi$. Since $|A^2_\phi(ww_1)|\ge 5$ and $|A^2_\phi(uu_2)|\ge 3$, these edges have at least one available color in common, say $2_1$, so we use this color on both $ww_1$ and $uu_2$. We now consider cases based on how $\phi$ can extend to $u_2u_{21}$ and $w_1w_{11}$. Our goal will be to extend the coloring $\phi$ to $G$ while coloring the edge set $H=\{uu_2, u_2u_{21}, vv_2, v_2v_{21}, xw, ww_1, w_1w_{11}\}$ with exactly 5 colors to make use of Lemma \ref{lem:4cut}. 
			
			\textbf{Case 3.1.} Suppose that $u_2u_{21}$ and $w_1w_{11}$ can receive distinct colors, so without loss of generality assume $u_2u_{21}$ gets color $2_2$ and $w_1w_{11}$ is gets color $2_3$. We proceed by coloring $vv_2$ with $2_\alpha \neq 2_1$ and $v_2v_{21}$ with $2_\beta \neq 2_1$ and consider cases based on $2_\alpha$ and $2_\beta$.
			
			First suppose that $2_\alpha, 2_\beta \not \in \{2_2, 2_3\}$, say $2_\alpha, 2_\beta \in \{2_4, 2_5\}$. Our goal is to color $xu$, $xv$, and $xw$ in such a way that we reuse some color in $\{2_1, 2_2, 2_3, 2_4, 2_5\}$ at $xw$, so that $H$ sees exactly 5 colors. Since $|A^2(xu)|\ge 2$, we can color $xu$ with a color $2\gamma \neq 2_\beta$. Since $xv$ now has at least one color available and $xv$ already sees $2_\beta$, we can color $xv$ with some color $2_\delta \neq 2_\beta$. Then we can color $xw$ with $2_\beta$, which finishes this case. 
			
			Now assume exactly one of $2_\alpha$ or $2_\beta$ belongs to the set $\{2_2, 2_3\}$. Note that the other color must be outside this set, say $2_4$, and the edges $xu$, $xv$, $xw$ have at least 2, 2, 4 colors available, respectively. So we can finish the coloring greedily in this order, ensuring that $xw$ receives color $2_\gamma \not \in \{2_1, 2_2, 2_3, 2_4\}$ so that the edges set $H$ receives 5 colors.
			
			Lastly, suppose that $2_\alpha, 2_\beta \in \{2_2, 2_3\}$. Since $|A^2_\phi (vv_2)|\ge 3$, then we must have $A^2_\phi (vv_2)=\{2_1, 2_2, 2_3\}$. We then uncolor the edges $vv_2$ and $ww_1$ and color the edge $vv_2$ with $2_1$. Since the edge $ww_1$ has at least 4 colors available, we can color $ww_1$ with color $2_\gamma \not \in \{2_1, 2_2, 2_3\}$, say $2_\gamma = 2_4$. As we plan to apply Lemma \ref{lem:4cut}, note that no color can be incident to all of the edges in $F'$. Since $xu$, $xv$, $xw$ have at least 2, 2, 4 colors available, respectively, we can finish the coloring greedily in this order, ensuring that $xw$ receives color $2_\delta \not \in \{2_1, 2_2, 2_3, 2_4\}$ so that the edges set $H$ receives 5 colors.
			
			\textbf{Case 3.2.} Now assume $u_2u_{21}$ and $w_1w_{11}$ must receive the same color, say $2_2$.
			Note that this implies that $w_{11}u_{21}\not \in E(G)$. Note that since $|A^2(vv_2)|\ge 3$, we can color $vv_2$ with color $2_\alpha \not \in \{2_1,2_2\}$, say $2_\alpha = 2_3$.
			
			Suppose we can color $v_2v_{21}$ with color $2_\beta \not \in \{2_1,2_2\}$, say $2_\beta=2_4$. Then since $|A^2(xu)|\ge 2$, $|A^2(xv)|\ge 2$, and $|A^2(xw)|\ge 4$, we can finish greedily while also ensuring that $xw$ receives color $2_\gamma \not \in \{2_1, 2_2, 2_3, 2_4\}$ so that the edges set $H$ receives 5 colors. 
			
			Otherwise, the edge $v_2v_{21}$ must receive color $2_1$ or $2_2$. Since $|A^2(v_2v_{21})|\ge 2$ we must have $A^2(v_2v_{21})=\{2_1,2_2\}$, so we color $v_2v_{21}$ with $2_1$. We uncolor $ww_1$ and $w_1w_{11}$ and color $w_1w_{11}$ with $2_1$. Since $|A^2(ww_1)|\ge 4$, we can color the edge $ww_1$ with color $2_\gamma \not \in \{2_1, 2_2, 2_3\}$, say $2_\gamma = 2_4$. As we plan to apply Lemma \ref{lem:4cut}, note now that no color can be incident to all of the edges in $F'$. Since $xu$, $xv$, $xw$ have at least 1, 2, 3 colors available, respectively, we can finish the coloring greedily in this order, ensuring that $xw$ receives color $2_\gamma \not \in \{2_1, 2_2, 2_3, 2_4\}$ so that the edges set $H$ receives 5 colors.
		\end{proof}
		
		
		\begin{newcase}
			$F=\{uu_1, u_2u_{21}, vv_1, ww_1\}$.
		\end{newcase}
		
		\begin{proof}
			We have $F'=\{u_2u_{22}, vv_2, ww_2\}$ and $M=\{x, u, u_2, v, w\}$ as in Figure \ref{fig:finalcases}(iv). By the minimality of $G$, both $G[L]$ and $G[R]$ have good colorings. Now impose the good coloring of $G[L]$ onto $G[L\cup M]$ and color $uu_1$ with a 1. We perform the adjustment on the colored edges incident to $u_{21}$, $v_1$, and $w_1$ so that each of the edges $u_2u_{21}$, $vv_1$, and $ww_1$ have at least 2 colors available. Our goal is to color the six edges $xv, xw, u_2u_{21}, uu_2, vv_1, ww_1$ with at most 5 colors, which would contradict Lemma \ref{lem:3cut}.  
			
			We begin by considering the colors available on $uu_2, vv_1, ww_1$. If at least two of these edges have an available color in common, say $2_1$, then we color the respective edges with $2_1$ and greedily color the remaining edge (if there is one). Since $u_2u_{21}, xu, xv,xw$ have at least 1, 3, 3, 3 colors available, respectively, we can finish greedily the coloring in that order if $|A^2(u_2u_{21})\cup A^2(xu) \cup A^2(xv) \cup A^2(xw)| \ge 4$. If $ |A^2(u_2u_{21})\cup A^2(xu) \cup A^2(xv) \cup A^2(xw)| = 3$ then we use the same color on $u_2u_{21}$ and then greedily color the rest two edges. 
			
			Otherwise, the edges $uu_2, vv_1, ww_1$ have no color in common. By the adjustment, this implies that $ww_1, vv_1, uu_2$ have exactly 2, 2, 3 colors available, respectively. This also means that $xv, xw$ have at least 6 colors available. Since $|A^2(u_2u_{21})|\ge 2$ and $|A^2(xw)|\ge 6$, then these edges must have an available color in common, say $2_1$. We proceed by coloring $u_2u_{21}$ and $xw$ with $2_1$. Since the edges $vv_1, ww_1, uu_2, xu, xv$ have at least 1, 1, 2, 4, 5 colors available, respectively, then we can finish the coloring greedily in this order. 
		\end{proof}   
		
		
		\begin{newcase} \label{config:6.3}
			$F=\{uu_1, u_2u_{21}, vv_1, w_1w_{11}\}$.
		\end{newcase}
		
		\begin{proof}
			We have $F'=\{u_2u_{22}, vv_2, w_1w_{12}, ww_2\}$ and $M=\{x, u, u_2, v, w, w_1\}$ as in Figure \ref{fig:finalcases}(v). By the minimality of $G$, both $G[L]$ and $G[R]$ have good colorings. Now impose the good coloring of $G[L]$ onto $G[L\cup M]$ and color $uu_1$ with a 1. We perform the adjustment on the colored edges incident to $u_{21}$, $v_1$, and $w_{11}$ so that each of the edges $u_2u_{21}$, $vv_1$, and $ww_1$ have at least 2 colors available. Call this good partial coloring $\phi$. Our goal will be to extend the coloring $\phi$ to $G$ while coloring the edge set $H=\{uu_2, u_2u_{21}, vv_1, xv, xw, ww_1, w_1w_{11}\}$ with exactly 5 colors to make use of Lemma \ref{lem:4cut}. 
			
			First note that $xw$ initially has 7 colors available, which will be important in our last case below. Since $|A^2_\phi (uu_2)| \ge 3$ and $|A^2_\phi (ww_1)| \ge 5$, the edges $uu_2$ and $ww_1$ have an available color in common, say $2_1$, so we color both these edges with $2_1$. Since $vv_1$ has at least 2 colors available, we color $vv_1$ with color $2_\alpha \neq 2_1$. Since the edges $u_2u_{21}$ and $w_1w_{11}$ both have at least 1 color available, we can color these edges with colors $2_\beta$ and $2_\gamma$, respectively. We now consider cases based on the colors $2_\alpha, 2_\beta, 2_\gamma$. 
			
			Suppose that $vv_1, u_2u_{21}, w_1w_{11}$ all receive the same color, say $2_2$. That is, $2_2=2_\alpha=2_\beta = 2_\gamma$. Then we uncolor $uu_2$ and since this edge has at least 2 colors available, we color $uu_2$ with some color not equal to $2_1$ or $2_2$, say we color $uu_2$ with $2_3$. The edges $xu, xv, xw$ have at least 2, 2, 4 colors available, respectively, so we can finish the coloring greedily in this order so that $H$ is colored with 5 colors.
			
			Suppose that only two of the edges $vv_1, u_2u_{21}, w_1w_{11}$ receive the same color. Then since $xu, xv, xw$ have at least 3, 3, 4 colors available (or 3, 4, 3, or 4, 3, 3, depending on which two of $vv_1, u_2u_{21}, w_1w_{11}$ receive the same color), we can finish the coloring greedily in this order so that $H$ is colored with 5 colors.

			Suppose that $vv_1, u_2u_{21}, w_1w_{11}$ receive distinct colors. Since $xw$ initially had all 7 colors available, we know that the color $2_\beta$ (used on $u_2u_{21}$) is available, so we color $xw$ with $2_\beta$. Then $xu, xv$ both have at least 2 colors available, so we can finish the coloring greedily. If $xv$ receives color $2_\delta$ that is not an element of the set $C = \{2_1, 2_\alpha, 2_\beta, 2_\gamma\}$, then we are done since we've used exactly 5 colors on the edges of $H$. Else, if $2_\delta \in C$, then we uncolor the edge $xw$. Since $xw$ initially had all 7 colors available, we can recolor $xw$ with some color not in $C$, thus ensuring $H$ is colored with 5 colors.
		\end{proof}
		
		Observe that $F=\{uu_1, u_2u_{21}, v_2v_{21}, ww_1\}$ is symmetric to Case \ref{config:6.3}.
		
		
		\begin{newcase}
			$F=\{uu_1, u_2u_{21}, v_2v_{21}, w_1w_{11}\}$.
		\end{newcase}
		
		\begin{proof}
			We have $F'=\{u_2u_{22}, vv_1, v_2v_{22}, w_1w_{12}, ww_2\}$ and $M=\{x, u, u_2, v, v_2, w, w_1\}$ as shown in Figure \ref{fig:finalcases}(vi). By the minimality of $G$, both $G[L]$ and $G[R]$ have good colorings. Now impose the good coloring of $G[R]$ onto $G[R\cup M]$ and color $vv_1$ and $ww_2$ with a 1. We perform the adjustment on the colored edges incident to $u_{22}$, $v_{22}$, and $w_{12}$ so that each of the edges $u_2u_{22}$, $v_2v_{22}$, and $w_1w_{12}$ have at least 2 colors available. Call this good partial coloring $\phi$. Our goal will be to extend the coloring $\phi$ to $G$ and color $H=\{uu_2, u_2u_{22}, xu, vv_2, v_2v_{22}, ww_1, w_1w_{12}\}$ with exactly 5 colors to make use of Lemma \ref{lem:4cut}. 
			
			First note that $xu$ initially has 7 colors available, which will be important in our cases below. Since $|A^2_\phi (uu_2)| \ge 5$ and $|A^2_\phi (ww_1)| \ge 3$, the edges $uu_2$ and $ww_1$ have an available color in common, say $2_1$, so we color both these edges with $2_1$.  Since $u_2u_{22}$ and $w_1w_{12}$ both have at least 1 color available, we can color these edges with colors $2_\beta$ and $2_\gamma$, respectively. Since $v_2v_{22}$ has at least 2 colors available, we color $v_2v_{22}$ with color $2_\alpha$. Note that $w_1w_{12}, u_2u_{22}, v_2v_{22}$ had at least $1,1,2$ colors available. Thus we may assume that $2_\alpha, 2_\beta, 2_\gamma$ are not all the same. We now consider cases based on the colors $2_\alpha, 2_\beta, 2_\gamma$. 
			
			Suppose that the edges $v_2v_{22}, u_2u_{22}, w_1w_{12}$ receive distinct colors, without loss of generality assume that either $2_\alpha =2_3, 2_\beta = 2_2, 2_\gamma = 2_4$ or $2_\alpha =2_1, 2_\beta = 2_2, 2_\gamma = 2_3$. \textbf{Case (A)}: $2_\alpha =2_3, 2_\beta = 2_2, 2_\gamma = 2_4$. Since $vv_2$ has at least 2 colors available, we can color it with some color not equal to $2_1$. If we can color $vv_2$ with some color from $\{2_2, 2_4\}$, then we do so; then $xw,xv,xu$ each has at least $2,2,3$ colors available. Thus, we can first color $xw,xv$ greedily in this order and then color $xu$ with a color from $\{2_5, 2_6, 2_7\}$ that is not used on $xw,xv$. Suppose now that we must color $vv_2$ with some color not in $\{2_1, 2_2, 2_3, 2_4\}$. Since $xu$ initially had all 7 colors available, we can color $xu$ with $2_3$; then $xw, xv$ has at least 1 and 2 colors available and we can finish by coloring $xw, xv$ greedily in this order. \textbf{Case (B)}: $2_\alpha =2_1, 2_\beta = 2_2, 2_\gamma = 2_3$. If we can color $vv_2$ with some color not in $\{2_1, 2_2, 2_3\}$, say $2_4$, then do so. We uncolor $uu_2$. Note that $uu_2$ has at least 4 colors available and $A_\phi ^2(uu_2)\subseteq \{2_1,2_3,2_4,2_5,2_6,2_7\}$, so at least one color belonging to $\{2_5,2_6,2_7\}$ is available at $uu_2$. We color $uu_2$ with a color in $\{2_5, 2_6, 2_7\}$. Since $xw,xv,xu$ have at least 1, 2, 3 colors available, respectively, we can $xw,xv, xu$ greedily in this order. Suppose now that we must color $vv_2$ with some color in $\{2_2, 2_3\}$. If we color $vv_2$ with $2_2$ ($2_3$).  We uncolor $uu_2$. Since $uu_2$ has at least $4$ colors available, we color it with a color in $\{2_4, 2_5, 2_6, 2_7\}$, say $2_4$. Then $xw, xv, xu$ has at least $1,2,4$ ($1,2,3$) colors available. We first color $xw,xv$ greedily in this order and then color $xu$ with a color in $\{2_5, 2_6, 2_7\}$ that is not used on $xw,xv$.
			
			Suppose that only two of the edges $v_2v_{22}, u_2u_{22}, w_1w_{12}$ receive the same color, without loss of generality we either have $2_\alpha, 2_\beta, 2_\gamma \in \{2_1, 2_2\}$ or $2_\alpha, 2_\beta, 2_\gamma \in \{2_2, 2_3\}$. \textbf{Case (A):} $2_\alpha, 2_\beta, 2_\gamma \in \{2_1, 2_2\}$. Then we must have $2_\alpha = 2_1$ and $2_\beta = 2_\gamma = 2_2$. Since $vv_2$ has at least 2 colors available, we can color $vv_2$ with some color not in $\{2_1, 2_2\}$, say $2_3$. We uncolor $uu_2$. Since $uu_2$ has at least $4$ colors available, we recolor $uu_2$ with a color not in $\{2_1, 2_2, 2_3\}$, say $2_4$. Then $xw, xv, xu$ has at least $1,2,3$ colors available. We can first color $xw, xv$ greedily and then color $xu$ with a color in $\{2_5, 2_6, 2_7\}$ that is not used on $xw, xv$. \textbf{Case (B):} $2_\alpha, 2_\beta, 2_\gamma \in \{2_2, 2_3\}$. If we can color $vv_2$ with some color not in $\{2_1, 2_2, 2_3\}$, say $2_4$, then $xw,xv,xu$ has $2, 2, 4$ colors available. We first color $xw, xv$ greedily in this order and then color $xu$ with a color in $\{2_5, 2_6, 2_7\}$ that is not used on $xw, xv$. Otherwise, suppose that $vv_2$ must receives some color from the set $\{2_1, 2_2, 2_3\}$. Then $v_2v_{22}$ must had at least $3$ colors available before we color $vv_2$. Thus, we may assume $2_{\gamma} \neq 2_\alpha$, say $2_\gamma = 2_2$ and $2_\alpha = 2_3$. The color $vv_2$ receives is in $\{2_1, 2_2\}$. Then we uncolor the edges $uu_2$, and since this edge has at least 4 colors available, we color $uu_2$ with some color not in $\{2_1, 2_2, 2_3\}$, say we color the edge $uu_2$ with $2_4$. Since $xv, xw, xu$ has at least $1,2,3$ colors available, we can color the edges $xv, xw$ greedily in this order and color $xu$ with a color in $\{2_5, 2_6, 2_7\}$ that is not used on $xv, xw$.
		\end{proof}
		
		
		\begin{newcase}
			$F=\{uu_1, vv_1, ww_1\}$.
		\end{newcase}
		
		\begin{proof}
			We have $F'=\{uu_2, vv_2, ww_2\}$ and $M=\{x, u, v, w\}$ as seen in Figure \ref{fig:finalcases}(vii). By the minimality of $G$, both $G[L]$ and $G[R]$ have good colorings. Now impose the good coloring of $G[L]$ onto $G[L \cup M]$. We perform the adjustment on the colored edges incident to $u_1$, $v_1$, and $w_1$ so that each of the edges $uu_1$, $vv_1$, and $ww_1$ have at least 2 colors available. Our goal is to color the six edges $xu, xv, xw, uu_1, vv_1, ww_1$ with at most 5 colors and make use of Lemma \ref{lem:3cut} on the cut $\{uu_2, vv_2,ww_2\}$.
			
			Suppose two of the edges $uu_1$, $vv_1$, $ww_1$ receive the same color, say both $uu_1$ and $vv_1$ are colored $2_1$. Then since $ww_1$, $xw$, $xv$, $xu$ have at least 2, 4, 4, 4 colors available, respectively, we are done by SDR. Thus to finish this proof, we will show in all cases that at least 2 of the edges $uu_1$, $vv_1$, $ww_1$ can receive the same color.
			
			Suppose at least two of the edges $uu_1$, $vv_1$, $ww_1$ each have at least 3 colors available, say $|A^2(uu_1)|\ge 3$ and $|A^2(vv_1)|\ge 3$. Then since $|A^2(uu_1)|\ge 3$, $|A^2(vv_1)|\ge 3$, and $|A^2(ww_1)|\ge 2$, at least 2 of the edges $uu_1$, $vv_1$, $ww_1$ have a color in common, so we are done by the above. 
			
			Otherwise, at least two of the edges $uu_1$, $vv_1$, $ww_1$ have at most 2 colors available, say $|A^2(uu_1)|\le 2$ and $|A^2(vv_1)| \le 2$. Since we used the adjustment on $u_1$ and $v_1$, we can say that $|A^2(uu_1)|= 2$ and $|A^2(vv_1)|= 2$. Then $|A^2(ux)| \ge 6$, $|A^2(vx)| \ge 6$, $|A^2(wx)|\ge 5$, and $|A^2(ww_1)| \ge 2$. Thus, we can color $ux$ and $ww_1$ with a common color, say $2_1$. Then $uu_1, vv_1, xw, xv$ have at least $1,1,3,4$ colors available and we can greedily color them in this order. 
		\end{proof}
		
		
		\begin{newcase}\label{config:9.2}
			$F=\{uu_1, vv_1, w_1w_{11}\}$.
		\end{newcase}
		
		\begin{proof}
			We have $F'=\{uu_2, vv_2, w_1w_{12}, ww_2\}$ and $M=\{x, u, v, w, w_1\}$ as shown in Figure \ref{fig:finalcases}(viii). By the minimality of $G$, both $G[L]$ and $G[R]$ have good colorings. Now impose the good coloring of $G[R]$ onto $G[R\cup M]$. We first color $ww_2$ with $1$. Then we perform the adjustment on the colored edges incident to $u_2$, $v_2$, and $w_{12}$ so that each of the edges $uu_2$, $vv_2$, and $w_1w_{12}$ have at least 2 colors available. Our goal is to color the six edges $xu,, uu_2, xv, vv_2, ww_1, w_1w_{12}$ with at most 5 colors, which contradicts Lemma \ref{lem:3cut}. 
			
			Consider the edges $uu_2$, $vv_2$, and $ww_1$, which have at least 2, 2, 3 colors available, respectively. Suppose two of these edges have a color in common. First, assume we can color $ww_1$ and $uu_2$ with $2_1$. Then we color $vv_2$ with some color, $2_\alpha$, and color $w_1w_{12}$ with some color $2_\beta$. Since $xw$, $xv$, $xu$ have at least 2, 3, 3 colors available, respectively, then we can finish the coloring greedily in this order. By symmetry, the case that we can use the same color on $vv_2$ and $ww_1$ follows similarly. Lastly, assume that we can color $uu_2$ and $vv_2$ with the same color, say $2_1$. Then we color $ww_1$ with some color, say $2_\gamma$, and color $w_1w_{12}$ with $2_\delta$. Then since $xw$, $xv$, $xu$ have at least 2, 3, 3 colors available, respectively, then we can finish the coloring greedily in this order.
			
			Otherwise, the edges $uu_2$, $vv_2$, and $ww_1$ have no available color in common. By the adjustment, this implies that $|A^2(uu_2)|=2$, $|A^2(vv_2)|=2$, and $|A^2(ww_1)|=3$, else there would be a color in common. Again by the adjustment, this means that $|A^2(xu)|\ge 6$, $|A^2(xv)|\ge 6$, and $|A^2(w_1w_{12})|\ge 3$. Since $|A^2(xu)|\ge 6$ and $|A^2(w_1w_{12})|\ge3$, we can color $xu$ and $ww_1$ with the same color, say $2_1$. Now $uu_2, vv_2, ww_1, xw, xv$ have at least $1, 1, 2, 4, 5$ colors available, respectively, we can finish the coloring greedily in this order. 
		\end{proof}
		
		Note that $F=\{uu_1, v_2v_{21}, ww_1\}$ is symmetric to Case \ref{config:9.2}.
		
		
		\begin{newcase} \label{config:10.2}
			$F=\{uu_1, v_2v_{21}, w_1w_{11}\}$.
		\end{newcase}
		
		\begin{proof}
			We have $F'=\{uu_2, v_2v_{22}, vv_1, w_1w_{12}, ww_2\}$ and $M=\{x, u, v, w, w_1\}$ as seen in Figure~\ref{fig:finalcases}(ix). By the minimality of $G$, both $G[L]$ and $G[R]$ have good colorings. Now impose the good coloring of $G[R]$ onto $G[R\cup M]$ and and color the edges $vv_1$ and $ww_2$ with 1. We perform the adjustment on the colored edges incident to $u_2$, $v_{22}$, and $w_{12}$ so that each of the edges $v_2v_{22}$, $uu_2$, and $w_1w_{12}$ have at least 2 colors available. Our goal is to color the six edges $xu,uu_2, vv_2, v_2v_{22}, ww_1, w_1w_{12}$ with at most 5 colors, which contradicts Lemma \ref{lem:3cut}. 
			
			First assume that $vv_2$ and $ww_1$ can get the same color, say $2_1$. Then we need only show that we can color the edges incident to vertices in $M$. Since $uu_2$ has at least two colors available,  we can color this edge some some color $2_\alpha \neq 2_1$. Since $v_2v_{22}$ and $w_1w_{12}$ have at least 1 color available, we can color these edges with some colors $2_\beta$ and $2_\gamma$, respectively. Since $xv, xw, xu$ have at least 2, 2, 3 colors available, respectively, then we can finish the coloring greedily in this order. 
			
			Now suppose that $vv_2$ and $ww_1$ cannot receive the same color. Then by the adjustment, this implies that $|A^2(vv_2)|=3$ or $|A^2(ww_1)|=3$, so without loss of generality assume that $|A^2(ww_1)|=3$. Then we must have $|A^2(w_1w_{12})| \ge 3$. Since $|A^2(xu)|\ge 5$ and $|A^2(w_1w_{12})| \ge 3$, then we can color these edges with the same color, say $2_1$. Since $uu_2$ still has at least 1 color available, we can color this edge with some color, say $2_2$. Now since we assumed $vv_2$ and $ww_1$ cannot receive the same color, we must have $|A^2(vv_2)\cup A^2(ww_1)| \ge 5$. Note that also $|A^2(xw)|\ge 3$ and $|A^2(v_2v_{22})|\ge 2$, so that if the edges $xw$ and $v_2v_{22}$ do not have an available color in common, then we are done by SDR. Otherwise, the edges $xw$ and $v_2v_{22}$ must have some color in common, say $2_3$. We color both $v_2v_{22}$ and $xw$ with $2_3$. Since $ww_1$, $vv_2$, $xv$ have at least 1, 1, 2 colors available, respectively, with $|A^2(vv_2)\cup A^2(ww_1)| \ge 4$, then we are done by SDR. 
		\end{proof}
		
		Note the cases where $F=\{uu_1, u_2u_{21}, vv_1, v_2v_{21}, ww_1\}$ and $F=\{u_2u_{21}, v_2v_{21}, ww_1\}$ are symmetric to Case \ref{config:10.2}.

		
		\begin{newcase}
			$F = \{u_2u_{21}, v_2v_{21}, w_1w_{11}\}$.
		\end{newcase}
		
		\begin{proof}
			Here we have $F' = \{u_2u_{22}, uu_1, v_2v_{22}, vv_1, ww_2, w_1w_{12}\}$ and $M = \{x, u, v, w, u_2, v_2, w_1\}$ as shown in Figure~\ref{fig:finalcases}(x).  By the minimality of $G$ both $G[L]$ and $G[R]$ have good colorings.  Impose the good coloring of $G[R]$ onto $G[R \cup M]$ and color $uu_1, vv_1$, and $ww_2$ with 1.  We then perform our adjustment to the colored edges incident to $u_{22}, v_{22}, w_{12}$.  Call this good partial coloring of $G[M \cup R]$, $\phi$. Our goal is to color $u_2u_{22}, uu_2, vv_2, v_2v_{22}, w_1w, w_1w_{12}$ with at most $5$ colors, which contradicts Lemma \ref{lem:3cut}. Observe that $|A^2_\phi(uu_2)|, |A^2_\phi(vv_2)|, |A^2_\phi(ww_1)| \ge 3$ so that without loss of generality, we can color $uu_2$ and $ww_1$ with a common color, say $2_1$.  
			
			Suppose we can color $v_2v_{22}$ with $2_1$ as well.  We then greedily color $u_2u_{22}$ and $w_1w_{12}$ (each had at least two colors available under $\phi$).  We then color $vv_2$ avoiding $2_1$ and the color used on $u_2u_{22}$ (this ensures that not all the edges relevant to Lemma \ref{lem:3cut} will only use two colors).  We can then finish greedily by coloring $xu, xw, xv$ in this order.
			
			So $2_1 \notin A^2_\phi(v_2v_{22})$.  Suppose we can color $vv_2$ with $2_1$.  We again color $u_2u_{22}$ and $w_1w_{12}$ greedily.  Similar to the above, we now color $v_2v_{22}$ with a color that is not used on $u_2u_{22}$ (this ensures that not all the edges relevant to Lemma \ref{lem:3cut} will only use two colors).  We then finish greedily by coloring $xu, xw, xv$ in this order.
			
			So $2_1 \notin A^2_\phi(vv_2)$ either.  Therefore no matter how we color the remaining uncolored edges, we will never end up with the edges relevant to Lemma \ref{lem:3cut} using only two colors.   Greedily color $u_2u_{22}, w_1w_{12}$ and call this good partial coloring $\psi$.  Observe that $|A^2_\psi(v_2v_{22})| \ge 2, |A^2_\psi(vv_2)| \ge 3$, $|A^2_\psi(xu)| \ge 3, |A^2_\psi(xw)| \ge 3$, and $|A^2_\psi(xv)| \ge 4$.  If we cannot color $v_2v_{22}$ and $xu$ the same, then we are done by SDR.  Otherwise, if we can color them the same, then we color them using a color in common and we are still done by SDR.
		\end{proof}

		\section{Further Research}\label{sec:further}
		
		There are several current conjectures regarding $S$-packing edge-colorings for subcubic graphs.  For example,  Hocquard et al. in \cite{HLL2} show that every 3-edge-colorable subcubic graph is $(1,2^7)$-packing edge-colorable, and pose the following conjecture.
		
		\begin{conjecture}[Hocquard, Lajou, and Lu\v zar~\cite{HLL2}]\label{hll2}
			Every $3$-edge-colorable subcubic graph is $(1, 2^6)$-packing edge colorable.
		\end{conjecture}

		In addition, Gastineau and Togni in \cite{GT2} showed that every subcubic graph $G$ with a $2$-factor has a $(1^2,2^5)$-packing edge-coloring, and if the graph is additionally 3-edge-colorable, then it is $(1^2,2^4)$-packing edge-colorable.  Hocquard et al. \cite{HLL2} were able to show the same holds without the assumption of a $2$-factor.   They also restate the following conjectures of Gastineau and Togni.

		\begin{conjecture}[Gastineau and Togni~\cite{GT2}, Hocquard et al. \cite{HLL2}]\label{gtedge}
			Every subcubic graph is $(1^2,2^4)$-packing edge colorable.
		\end{conjecture}

		\begin{conjecture}[Gastineau and Togni~\cite{GT2}, Hocquard et al. \cite{HLL2}]\label{gtedge2}
			Every $3$-edge-colorable, subcubic graph is $(1^2,2^3)$-packing edge colorable.
		\end{conjecture}

		\section{Appendix}
		
		The proofs of Lemmas \ref{nine} and \ref{lem:4cut} involve computer computations utilizing SAGE.  The explanations and full code for each can be found at the following links.
		
		Lemma \ref{nine}:  \href{https://sites.google.com/site/santanagvsu/sage/127-packing-edge-coloring/lemma-9?authuser=0}{https://sites.google.com/site/santanagvsu/sage/127-packing-edge-coloring/lemma-9?authuser=0}

		Lemma \ref{lem:4cut}:  \href{https://sites.google.com/site/santanagvsu/sage/127-packing-edge-coloring/lemma-21?authuser=0}{https://sites.google.com/site/santanagvsu/sage/127-packing-edge-coloring/lemma-21?authuser=0}

	\end{document}